\documentclass[10pt]{article} 
\usepackage[utf8x]{inputenc}
\usepackage{amssymb}
\usepackage{amsmath}
\usepackage{amsfonts}
\usepackage{amscd}
\usepackage[all,cmtip]{xy}
\usepackage{makeidx}
\usepackage{titlesec}
\usepackage{fancyhdr}
\usepackage[colorlinks]{hyperref}
\usepackage [pdftex]{graphicx}

\usepackage{thmtools, thm-restate}  
\usepackage [pdftex]{color}
\usepackage{tikz}
\usepackage{fancyhdr}
 \usepackage{latexsym}
\usepackage{geometry,enumerate,latexsym,tabularx}
\usepackage{multicol}
\usepackage{wrapfig}
\usepackage{stmaryrd}
\usepackage[cm]{fullpage}
\usepackage{layout}
\usepackage{eqnarray}
\usepackage{array}

\def\Aff{\mathrm{Aff}} 
\def\bH{\mathbb{H}} 
\def\bS{\mathbb{S}} 
\def\bT{\mathbb{T}} 
\def\cF{\mathcal{F}} 
\def\C{\mathbb{C}}
\def\H{\mathbb{H}}
\def\Im{\mathrm{Im}}
\def\N{\mathbb{N}}

\def\Q{\mathbb{Q}}  

\def\PSL{\mathbf{PSL}}
\def\R{\mathbb{R}}
\def\Re{\mathrm{Re}}
\def\S{\mathbb{S}}
\def\SL{\mathbf{SL}}
\def\SO{\mathbf{SO}}
\def\SU{\mathbf{SU}}
\def\Z{\mathbb{Z}}

\titleformat{\section}[hang]
{\normalfont\filright\large}{\thesection. }{0pt}
{\upshape\bfseries}

\titleformat{\subsection}[hang]
{\itshape}{\thesubsection \ - }{0pt}
{}

\usepackage{amsthm}
\theoremstyle{plain}
\newtheorem{theo}{Theorem}[section]
\newtheorem{prop}[theo]{Proposition}
\newtheorem{lemm}[theo]{Lemma}
\newtheorem{coro}[theo]{Corollary}
\newtheorem{rem}[theo]{Remark}
\newtheorem*{tio}{Theorem}
\theoremstyle{remark}

\theoremstyle{definition}
\newtheorem{defi}[theo]{Definition}
\newtheorem*{Example-non}{Example}
\newtheorem*{Lemma-non}{Lemma}
\newtheorem*{Proposition-non}{Proposition}
\newtheorem*{Corollary-non}{Corollary}
\newtheorem*{Definition-non}{Definition}

\def\ie{\emph{i.e.\,}}

\providecommand{\abs}[1]{\lvert#1\rvert}
\providecommand{\norm}[1]{\lVert#1\rVert}

\title{Eisenstein series and equidistribution of Lebesgue probability measures on compact leaves of the horocycle foliations of  Bianchi 3-orbifolds}
\author{\small Otto Romero \footnote{Instituto de Matem$\acute{\mbox{a}}$ticas, (Unidad Cuernavaca) UNAM. Av. Universidad s/n. Col. Lomas de Chamilpa CP 62210, Cuernavaca, Mexico. \texttt{alberto@matcuer.unam.mx}}\,\,\,, Alberto Verjovsky  \footnote{Instituto de Matem$\acute{\mbox{a}}$ticas, (Unidad Cuernavaca) UNAM. Av. Universidad s/n. Col. Lomas de Chamilpa CP 62210, Cuernavaca, Mexico. \texttt{ottohrg@gmail.com}}}

\date{\today}
\begin{document}

\maketitle

\begin{abstract}
Inspired by the works of  Zagier \cite{Zag79} and Sarnak \cite{Sar80},  we study the probability measures
$\nu(t)$ with support on the flat tori which are the compact orbits of the maximal unipotent subgroup acting holomorphically on the positive orthonormal frame bundle $\mathcal{F}({M}_D)$ of 3-dimensional hyperbolic Bianchi orbifolds ${M}_D=\bH^3/\widetilde{\Gamma}_D$, of finite volume and with only one cusp. Here 
$\widetilde{\Gamma}_D\subset\PSL(2,\C)$ is the  Bianchi group corresponding to the
imaginary quadratic field $\Q(\sqrt{-D})$. Thus $\widetilde{\Gamma}_D$ consist of M\"obius transformations  with coefficients in the ring of integers of $\Q(\sqrt{-D})$.
If $l \in \N$, $k,m \in \Z$  are such that $k,m \in [-l,l]$, 
the appropriate Eisenstein series $\widehat{E}_{km}^l ( g,s)$ (which are defined and analytic for $\Re (s) \geq 1$) 
admit an  analytic continuation to all of $\C$, except when $l=k=m=0$ in which case there is a pole for $s=1$. Using this fact we show that the elliptic curves which are the compact orbits of the complex horocycle flow $h_T:\mathcal{F}({M}_D)\longrightarrow\mathcal{F}({M}_D)$  ($T \in \C$) are expanded by the real geodesic flow $g_t, t\in\R$, and they become equidistributed in  $\mathcal{F}({M}_D)$ with respect to the normalized Haar measure as $t \longrightarrow\infty$. This follows from the
equidistribution of Lebesgue probability measures on compact leaves of the horocycle foliations in the orthonormal frame bundle of $M_D$, which is equal to the quotient
$\PSL(2,\C)/\widetilde{\Gamma}_D$. The same equidistribution property occurs for the spin bundle of $M_D$ which is the homogeneous space $\SL(2,\C)/\Gamma_D$, where $\Gamma_D$ is the Bianchi subgroup in $\SL(2,\C)$ 
consisting of matrices with entries in the ring of integers of $\Q(\sqrt{-D})$.
Our method uses the theory of spherical harmonics in the unit tangent bundle orbifold  $T_1(M_D)={\SO}(2){\backslash}\PSL(2,\C)/\widetilde{\Gamma}_D=T_1(\bH^3)/{\Gamma^\ast_D}$,
where ${\Gamma^\ast_D}$ is the action of $\widetilde\Gamma_D$ on the unit tangent bundle of $\H^3$ via the differential of the elements of $\widetilde{\Gamma}_D$. 



\end{abstract}

\setcounter{tocdepth}{5}
\tableofcontents

\section*{Introduction}
The theory of equidistribution of probability measures supported on compact horocycles has been studied by many authors, for instance, Bowen and Marcus \cite{BM}, Dani \cite{Da, DS}, Flaminio and Forni \cite{FF}, Furstenberg \cite{Fu}, Ratner \cite{Ra}, Selberg, Smillie \cite{DS}, Sarnak \cite{Sar80}, Zagier \cite{Zag79}, among many others. This study is done using different techniques and viewpoints but in essence is the study of unipotent flows in homogeneous spaces of the form $G/\Gamma$, where
$G$ is an algebraic group and $\Gamma$ is a discrete subgroup such that $G/\Gamma$ has finite Haar measure.

In this paper we study the case when $G=\PSL(2,\C)$ and $\Gamma$ is a Bianchi subgroup of certain imaginary quadratic fields. 

\bigskip
More precisely, let $D$ be a square-free positive integer. Let $\Q(  \sqrt{-D})$  be the corresponding imaginary quadratic field
and $\mathcal{O}_D$ its ring of integers. Let $\widetilde{\Gamma}_D = \PSL(2,\mathcal{O}_D)$ be the corresponding Bianchi subgroup in $\PSL(2,\C)$ and $\Gamma_D =\SL(2,\mathcal{O}_D)$ the  Bianchi subgroup in $\SL(2,\C)$. The quotients ${M}_D= \H^3 /  \widetilde{\Gamma}_D$ are the {3-dimensional Bianchi orbifolds}.
Since $\PSL(2,\C)$ acts, via the differential, simply and transitively on the orthonormal frame bundle (with respect to the hyperbolic metric) of positively oriented frames of hyperbolic 3-space $\H^3$ it follows that $\PSL(2,\C)$ can be identified with the orthonormal frame bundle  $\mathcal{F}(\H^3)$ of hyperbolic 3-space, and the space of left cosets can be identified with the frame bundle $\mathcal{F}({M}_D)$ of ${M}_D$.

\bigskip
The group $\PSL(2,\C)$ acts \emph{on the left} on the quotient space of left cosets
$\mathcal{F}({M}_D)=\PSL(2,\C)/\widetilde{\Gamma}_D$. This is a non-K\"ahler complex smooth 3-fold.
The complex affine group $\Aff(\C)\subset\PSL(2,\C)$ consisting of M\"obius transformations of the form
$z\mapsto{az+b},\quad a\in\C^*, b\in\C$,  acts on the left, holomorphically and locally freely on the frame bundle
$\mathcal{F}({M}_D)$. The affine group corresponds to upper triangular matrices

$$
\left( \begin{matrix}
e^{\frac{T}2} & S \\
0 & e^{-\frac{T}2}  \\
\end{matrix}\right),\quad \quad T, S \in\C.
$$
This action defines two holomorphic flows
$g_T:\mathcal{F}({M}_D)\longrightarrow\mathcal{F}({M}_D)$ and
$h_S:\mathcal{F}({M}_D)\longrightarrow\mathcal{F}({M}_D)$ ($T,S\in\C$) determined by the diagonal matrices and unipotent matrices, respectively. They are the \emph{complex frame geodesic flow} and
\emph{complex frame horocycle flow}, respectively. These flows are the orbits of two holomorphic vector fields $X$ and $Y$ on $\mathcal{F}({M}_D)$ satisfying the Lie bracket relation $[X,Y]=Y$ (they are Killing vector fields of the locally free action of $\Aff(\C))$.
This bracket relation is equivalent to $dg_{_T}\circ{Y}\circ{g_{_{-T}}}:=\left(g_{_T}\right)^*(Y)=e^TY$. Therefore $|\left(g_{_T}\right)^*(Y)|=e^{\Re \, T}|Y|$    ($T=\Re \, T+i\Im\, T$) and $|\left(g_{_t}\right)^*(Y)|=e^{t}|Y|, t\in\R$, and $|\left(g_{_t}\right)^*(Y)|$ tends exponentially to infinity as $t\longrightarrow+\infty$.  Thus the flow $\left\{g_T\right\}_{T\in\C}$ is a holomorphic Anosov flow (see definition in \cite{Gh} and \cite{GV}).

\bigskip
The orbits of $\left\{g_T\right\}_{T\in\C}$ are either immersed copies of $\C^*$ or elliptic curves. There is a countable number of elliptic curves which correspond to the closed geodesics of the orbifold ${M}_D$. The moduli of the elliptic curves depends on the length of the corresponding geodesic. By a complex version of a theorem by Dani \cite{Da} every orbit of the complex horocycle flow $\left\{h_S\right\}_{S\in\C}$  is either a densely embedded copy of $\C$ or an elliptic curve and all the orbits which are elliptic curves are of the form $g_T(E)$, $T\in\C$ where $E$ is an elliptic curve of area whose modulus depends only on a unit of the imaginary quadratic field (we are assuming that the orbifold ${M}_D$ has only one cusp).
\emph{In fact the only ergodic invariant measures for the complex horocycle flow are either the probability Lebesgue measures on the elliptic curves or normalized Liouville measure}.\\

We will construct series $\widehat{E}_{km}^l ( g,s)$ which play the role of Eisenstein series in our context.
The meromorphic continuation of this series (theorem (\ref{conje}))
implies the following result in dimension 3, which is analogous to theorem $1$ of Sarnak in
 \cite{Sar80} for dimension 2.
\begin{theo}\label{TPT}
Let $D \in  \{ 1,2,3,7,11,19,43,67,163 \}$ and $M_D= \H^3 /  \Gamma_D$ the Bianchi hyperbolic 3-orbifolds  with only one cusp (\emph{i.e.} determined by an imaginary quadratic field of class number one). Then there exist numbers $1 > s_1 > s_2 > \cdots > s_p > 0$ ($p$ could be zero, in which case such numbers don't appear) and corresponding distributions  $\nu^1 , \nu^2 , \cdots , \nu^p $ in the unit tangent bundle orbifold $T_1 (M_D)$,  such that if $f: T_1 (M_D) \longrightarrow
\C$ is a smooth function with compact support then:
\[
      \nu (\lambda)(f) \; = \;  \nu (f) \, + \, \lambda^{1-s_1} \, \nu^1 (f)  +  \cdots +  \lambda^{1-s_p} \, \nu^p  (f) +
          \, o(\lambda)  \; \; \; \; \; (\lambda \longrightarrow 0),
\]
where $\nu= \frac{1}{4 \pi \, \text{Vol} ( M_D)} \, \frac{\sin \vartheta \,  dx  dy  d\lambda  d\vartheta  d\varphi}{\lambda^3}$
is normalized Liouville measure in $T_1 (M_D)$. 
\end{theo}
\begin{Corollary-non}
\[
   \displaystyle  \lim_{\lambda  \longrightarrow 0 } \nu(\lambda) \; = \;  \nu   \; \; \; \; \text{(weakly). }
\]

\end{Corollary-non}


\subsection*{Outline of the paper}
In section 1 we review the action of
 $\SL(2,\C)$ on $\H^3$,  Iwasawa decomposition and Iwasawa coordinates.   We recall some properties of Bianchi orbifolds of class number one,  the stabilizer of the point at infinity and the maximal unipotent subgroup.

\bigskip
In section  2 we study some properties of the Wigner matrix $D_{km}^l \big( R  \big)$ with $R \in \SO(3)$.

\bigskip
In section 3 we will define rotations $R(d\sigma,P) \in \SO(3)$, where $\sigma \in \SL(2,\C)$ and $P \in \H^3$. We will define and study Eisenstein series $\widehat{E}_{km}^l (g,s)$  and their ``projections'', given by series  $E_{km}^l (P_g,s)$. We will analyze their convergence, invariance, Fourier expansions and their analytic or meromorphic continuation.

\bigskip
In section 4 we review about the geodesic flow and horocycle foliations.

\bigskip
Finally in section 5 we prove theorem (\ref{TPT}) using the meromorphic continuation of Eisenstein series $\widehat{E}_{km}^l (g,s)$ (see theorem (\ref{conje})).


\section{Preliminaries: Hyperbolic geometry,
Bianchi groups and Bianchi orbifolds.}

We will use the upper half-space model for hyperbolic space  $ \H^3 $:
\[
  \H^3  \; = \; \{ (z, \lambda) \, ; \, z \in \C , \lambda \in \R^+ \} =
              \{ (x,y, \lambda) \, ; \, x,y \in \R , \lambda > 0 \},
\]
provided with the hyperbolic Riemannian metric:

\[
       ds^2 \; = \; \frac{dx^2+dy^2 + d\lambda^2}{\lambda^2}.
\]
We will denote a point $P \in \H^3$ as follows:
\[
    P \; = \; (z,\lambda) = (x,y,\lambda) = z + \lambda j,
\]
where $z=x+iy$, $\lambda>0$, and $j$ is the well know quaternion.  We recall that
\[
   \displaystyle   \SL(2,\C) \; = \; \Bigg \{ \left(  \begin{matrix}
a & b \\
c & d \\
\end{matrix}  \right) ; a,b,c,d \in \C, \; \; ad-bc=1    \Bigg \}
\]
and
\[
      \PSL( 2,\C) \; = \; \SL(2,\C) / \{\pm I\}.
\]

The group $\SL(2,\C)$ acts on $\H^3$ as follows:
\begin{equation}\label{Pre3}
    \displaystyle  \left(  \begin{matrix}
a & b \\
c & d \\
\end{matrix}  \right)  ( z + \lambda  j ) \, =  \,
    \frac{ (a z  + b)(\bar{c} \bar{z} + \bar{d}) + a \bar{c} \, \lambda^2 }{  \vert cz + d \vert^2 + \vert c \vert^2 \, \lambda^2 } +
         \frac{\lambda}{ \vert cz + d \vert^2 + \vert c \vert^2 \, \lambda^2 }   \, j.
\end{equation}

\begin{rem}The action is not effective since two matrices which differ by a sign determine the same M\"obius transformation.The group  $\PSL( 2,\C)$ acts effectively on $ \H^3 $
 by orientation-preserving isometries of the hyperbolic metric. If $\Gamma$ is a discrete subgroup of $\SL(2,\C)$
 and $\widetilde\Gamma\subset\PSL(2,\C)$ is the corresponding group of M\"obius transformations 
 then $M_{\Gamma}:=\H^3/\Gamma=\H^3/\widetilde\Gamma:=M_{\widetilde{\Gamma}}$, so sometimes we use one or the other quotient to denote this manifold.
\end{rem}

\begin{lemm}\label{DIwasawa}
(Complex Iwasawa decomposition)
If $g \; = \; \left(  \begin{matrix}
a & b \\
c & d \\
\end{matrix}  \right) \, \in \, \SL(2,\C)$ then $g$ can be written in a unique way by the following formula:
\[
    \displaystyle \left(  \begin{matrix}
a & b \\
c & d \\
\end{matrix}  \right) = \left(  \begin{matrix}
1 & \frac{ a \bar{c} +  b \bar{d} }{ \vert c \vert^2 + \vert d \vert^2 }\\
0 & 1 \\
\end{matrix}  \right)  \left(  \begin{matrix}
\frac{ 1}{ \sqrt{ \vert c \vert^2 + \vert d \vert^2 } } & 0 \\
0 & \sqrt{ \vert c \vert^2 + \vert d \vert^2 } \\
\end{matrix}  \right)   \left(  \begin{matrix}
\frac{ \bar{d}}{ \sqrt{ \vert c \vert^2 + \vert d \vert^2 } } & \frac{ - \bar{c}}{ \sqrt{ \vert c \vert^2 + \vert d \vert^2 } } \\
\frac{ c}{ \sqrt{ \vert c \vert^2 + \vert d \vert^2 } } & \frac{ d}{ \sqrt{ \vert c \vert^2 + \vert d \vert^2 } } \\
\end{matrix}  \right).
\]
\end{lemm}

Hence $g$ is determined by the triple $(z,\lambda, A)$ with $z\in\C$, $\lambda>0$ and $A\in{\SU(2)}$. We call  $(z,\lambda, A)$ the \emph{Iwasawa coordinates} of $g$. One has:

\begin{equation}\label{Pre4}
    \displaystyle  z= \frac{ a \bar{c} +  b \bar{d} }{\vert c \vert^2 + \vert d \vert^2 },  \; \;  \; \; \;
    \lambda = \vert c \vert^2 + \vert d \vert^2, \; \; \; \; \; A = \left(  \begin{matrix}
\frac{ \bar{d}}{ \sqrt{ \lambda } } & \frac{ - \bar{c}}{ \sqrt{ \lambda } } \\
\frac{ c}{ \sqrt{ \lambda } } & \frac{ d}{ \sqrt{ \lambda } } \\
\end{matrix}  \right) \in \SU(2).
\end{equation}
Using Iwasawa coordinates we can write  $g$ as follows:
\[
       g \; = \; n[z] \, a[\tfrac{1}{\lambda}] \, A, 
\]
where
\[
n[z] := \left(  \begin{matrix}
1 & z \\
0 & 1 \\
\end{matrix}  \right) \hspace{4cm}
a[\lambda] := \left(  \begin{matrix}
\sqrt{\lambda} & 0 \\
0 & \frac{1}{\sqrt{\lambda}} \\
\end{matrix}  \right). 
\]


\bigskip
A useful formula (a particular case of formula (\ref{Pre3})) is the following:
\begin{equation}\label{Pre6}
    \displaystyle \big(  n[z] \, a[\lambda] \, A \big) (j) \; = \; 
    z + \lambda j.
\end{equation}

The Lie group $\SU(2)$ is diffeomorphic to the 3-sphere $\S^3$. Iwasawa decomposition implies that $\SL(2,\C)$ is diffeomorphic to $\H^3 \times \SU(2)\cong\H^3 \times\S^3 $.

\bigskip
Let $T: \SL(2,\C) \longrightarrow \SU(2)$ be the function that assigns the part in $\SU(2)$ of the Iwasawa decomposition. In other words, if $g = n[z] \, a[\lambda] \, K \in \SL(2,\C)$,
then $T(g) \; := \; K$.

\bigskip
One has the following theorem (\cite{Els98}, page $311$).
\begin{theo}\label{TeoGroups1}
Let $K_D$ be an imaginary quadratic field with discriminant $d_D<0$. Let $\mathcal{O}_D$ be its ring of integers.
Let us consider the groups: 
\[
   \displaystyle  \Gamma_D:= \SL(2,\mathcal{O}_D) \; = \; \Bigg \{ \left(  \begin{matrix}
a & b \\
c & d \\
\end{matrix}  \right) ; a,b,c,d \in\mathcal{O}_D, \; \; ad-bc=1    \Bigg \}
\]
and
\[
     \widetilde{\Gamma}_D:= \PSL( 2,\mathcal{O}_D) \; = \; \SL(2,\mathcal{O}_D) / \{\pm I\}.
\]

Thus $\widetilde{\Gamma}_D= \PSL(2,\mathcal{O}_D)$ consists of M\"obius transformations $z\mapsto\frac{az+b}{cz+d}$ with $a,b,c,d \in\mathcal{O}_D$.
Then the groups ${\Gamma}_D$ and $\widetilde{\Gamma}_D$ have the following properties:
\begin{enumerate}
  \item   ${\Gamma}_D$ and $\widetilde{\Gamma}_D$ is a discrete subgroup of  $\SL(2,\C)$ and 
  $\PSL(2,\C)$, respectively. They are both finitely presented.
  \item $\widetilde{\Gamma}_D$ is not co-compact but it has finite co-volume. 
  \item
        \[
            \text{Vol} \big(  \H^3 / \widetilde{\Gamma}_D   \big)   \; = \;  \frac{\abs{d_D}^{\frac{3}{2}}}{4 \pi^2} \, \zeta_{D} (2), \quad \text{where}  \,\,  \zeta_{D} \,\text{is the  Dedekind zeta function associated to}  \,K_D.
        \]
  \item $\widetilde{\Gamma}_D$ has a fundamental domain bounded by a finite number of totally geodesic surfaces 
  $\mathcal{F}_D$ (\ie it is geometrically finite).
\end{enumerate}
\end{theo}

The previous theorem implies that for each square-free positive integer $D$:
\[
     {M}_D \; = \;  \H^3 /  \widetilde{\Gamma}_D \; = \;  \H^3 /  \Gamma_D,
\]
is a non-compact  hyperbolic 3-orbifold of finite volume. 
$M_D$ is called \emph{Bianchi orbifold corresponding to $K_D$}, \cite{Bianchi}.

\bigskip
Let  $\infty  \in  \H^3 \cup \mathbb{P}^1 (\C)$ and $\Gamma_{D, \infty}$ the stabilizer of  $\infty$ in $\Gamma_D$. In other words
\[
     \Gamma_{D, \infty} =  \Bigg\{
    \left(
\begin{matrix}
\epsilon & z \\
0 & \epsilon^{-1} \\
\end{matrix}  \right) \;  ; \; \epsilon  \in \mathcal{O}^{\times}_D , \; z \in \mathcal{O}_D  \Bigg\},
\]
where $\mathcal{O}^{\times}_D$ denotes unit group of units of $\mathcal{O}_D$.

\bigskip
The maximal unipotent subgroup  $ \Gamma'_{D, \infty}$ of $\Gamma_D$ at $\infty$ is given by:
   \[
     \Gamma'_{D, \infty} =  \Bigg\{
    \left(
\begin{matrix}
1 & z \\
0 & 1 \\
\end{matrix}  \right) \;  ; \; z \in \mathcal{O}_D  \Bigg\}.
\]

Since $D>0$,  $ \Gamma_{D, \infty}$ contains a free abelian group of rank 2, which implies that $\infty$ defines a cusp of $\Gamma_{D}$. It is a classical result that was proved by Bianchi \cite{Bianchi} and Hurwitz that the number of cusps of $\Gamma_D$ is equal to the class number of $K_D$. As $D> 0$, the Baker--Heegner--Stark theorem 
(see \cite{Stark1, Stark2}) says that the class number of $K_D$ is equal to $1$ if and only if
$D \in  \{\,1,2,3,7,11,19,43,67,163\,\}$. Therefore, for those $D$,  $M_D= \H^3 /  \Gamma_D$  are hyperbolic 3-orbifolds with a unique cusp in $\infty$.

\subsection*{Notation}

Let  $D \in \N$ be square-free integer. Let $K_D= \Q (\sqrt{-D})$ be the corresponding imaginary quadratic field. We will denote by $\mathcal{O}_D$  its ring of integers and by
 $\mathcal{O}^{\times}_D $ its group of units. In order not to overload the notation we will sometimes refer to $\Gamma_D$ simply as $\Gamma$, $\Gamma_{D, \infty}$ as $\Gamma_{\infty}$ and $\Gamma'_{D, \infty}$ as $\Gamma'_{\infty}$. 

\section{Spherical harmonics and the Wigner matrix}
Consider the Laplacian on the circle $\S^1$,   $\triangle_{\S^1}= - \frac{\partial^2}{\partial \theta^2}$. Let $n \in \Z$ and $\psi_n : \S^1 \longrightarrow \C$ its eigenfunctions \emph{i.e.,} $\psi_n (z) = z^{n}$, with $\abs{z}=1$. One has,
\[
       \triangle_{\S^1} \psi_n \; = \; n^2 \psi_n,
\]
hence, $\text{Spec} (\triangle_{\S^1}) = \{ n^2 \, ; \, n \in \Z   \}$. Therefore, it has multiplicity $2$ for $n \neq 0$. Let $R_\alpha \in \SO(2)$ a rotation of angle $\alpha \in [0,2\pi)$ and consider the function $\psi_n \circ R_\alpha : \S^1 \longrightarrow \C$. Hence $\psi_n \circ R_\alpha (z) \, = \, e^{in\alpha} \, \psi_n (z)$. For each $n \in \Z$ there exists a  representation $\Psi_n : \SO(2) \longrightarrow M_{1 \times 1} (\C) \cong \C$ given by $  \Psi_n ( R_\alpha ) := e^{in \alpha}$. Therefore,
\begin{equation}\label{LDep1}
  \psi_n \circ R_\alpha \; = \;   \Psi_n ( R_\alpha ) \cdot \psi_n.
\end{equation}
In other words,  $\psi_n \circ R_\alpha$ and $\psi_n$ are linearly dependent eigenfunctions in $\S^1$.  The coefficients that appears in the linear combination of such functions come from the representations of $\SO(2)$.

\bigskip
Let us consider now the Laplacian on $\S^2$, denoted $\triangle_{\S^2}$. Let $l \in \N$, $m \in \Z$ with $ m \in [-l,l]$. The spherical harmonics $Y_m^l : \S^2 \longrightarrow \C$ are the eigenfunctions of $\triangle_{\S^2}$, \emph{i.e.,}
\[
    \displaystyle  \triangle_{ \S^2 } Y_m^l  \; = \; l (l + 1) \, Y_m^l,
\]
therefore, $\text{Spec} (\triangle_{\S^2}) = \{ l(l+1)  \, ; \, l \in \N \cup  \{0\} \}$. Now the multiplicity grows with $l$. Let $R \in \SO(3)$, and $Y_m^l \circ R : \S^2 \longrightarrow \C$. The functions  $\{ Y_m^l \circ R, Y_k^l \}_{k=-l}^l$ are linearly dependent eigenfunctions of $\S^2$. We can write $R$ in terms of its Euler angles:
\[
      R \; = \; ROT(\theta, \chi, \phi) \; : = \;
      \left(  \begin{matrix}
 \cos \theta   & \sin \theta  & 0  \\
-\sin \theta   & \cos \theta  & 0   \\
0 & 0 &   1   \\
\end{matrix}  \right)
 \left(  \begin{matrix}
 \cos \chi   & 0   & - \sin \chi  \\
0    & 1  & 0   \\
 \sin \chi  & 0 &   \cos \chi   \\
\end{matrix}  \right)
  \left(  \begin{matrix}
 \cos \phi   & \sin \phi  & 0  \\
-\sin \phi   & \cos \phi  & 0   \\
0 & 0 &   1   \\
\end{matrix}  \right),
\]
with  $\theta \in [0,2\pi), \chi \in [0,\pi], \phi \in [-2\pi,2\pi)$. There is a formula analogous to (\ref{LDep1}):
\begin{equation}\label{SumatoriaWigner}
       Y^l_{m} \Big( R^{-1} (\vartheta, \varphi) \Big) \; = \;     \sum_{k=-l}^{l} D^l_{km} (R) \cdot  Y^l_{k} (\vartheta, \varphi).
\end{equation}
The pair $(\vartheta, \varphi)$ with  $\vartheta \in [0,\pi]$, $\varphi \in [0,2\pi)$, gives the spherical coordinates of the point $p=(x,y,z)\in\R^3$ given by
\begin{align*}
  x \; = \; &   \cos \varphi \cdot   \sin \vartheta, \\
  y \; = \; &   \sin \varphi \cdot \sin \vartheta, \\
  z \; = \; &   \cos \vartheta,
\end{align*}
in addition one has,
\begin{equation}
\label{WignerIg}
     \displaystyle  D_{km}^l \big( R  \big) \; = \; D_{km}^l \big( ROT(\theta, \chi, \phi)  \big) = e^{i k \theta } \, d_{km}^l (\chi) \,  e^{i m \phi},
\end{equation}
where $d_{km}^l (\chi)$ is the \emph{Wigner small $d$-matrix}. 

\bigskip
We refer to the book of Wigner \cite{Wigner}, to see how the coefficients $D_{km}^l \big( R  \big)$ are deduced. These coefficients come from $(2l+1)$-dimensional representations  of $\SO(3)$ (or its universal cover $\SU(2)$). This is in complete analogy with the case of the circle
 $U(1)=\S^1$.


\section{Eisenstein series}
\subsection{Preliminaries}
Let $P=z + \lambda j \in  \H^3$, we will define a matrix $R(d\sigma,P) \in M_{3 \times 3 } (\R )$, where $\sigma \in \SL(2,\C)$. This matrix is completely determined by the value of the differential of $\sigma$ en $P$ on a Euclidean orthonormal basis $\{e_1,e_2,e_3\} $ of $T_P \R^3$ \emph{i.e.,} it is determined by the orthogonal basis  $\{d \sigma_P \, e_1,  d \sigma_P \, e_2, d \sigma_P \, e_3\}$. More explicitly      
\begin{equation}\label{Rota9}
   \displaystyle R(d\sigma,P) \, v  \; := \;
     \frac{\lambda}{\Im \, \sigma P} \; d \sigma_P (v), \; \; \;  \forall v \in T_P \R^3.
\end{equation}

The lineal transformation $R(d\sigma,P): \R^3 \longrightarrow \R^3$  preserves Euclidean norm.

\bigskip
In  dimension 2, if  $\sigma \in \SL(2,\R)$  and $z \in \H^2$ then $d\sigma_{z}$ acts on
$\R^2$ as the composition of a rotation in $\SO(2)$ and a homothety. The rotation $R(d\sigma,P) \in \SO(3)$  is the analogous in dimension $3$.

\begin{lemm}\label{Series5}
Let $\sigma= \left(  \begin{matrix}
a & b \\
c & d \\
\end{matrix}  \right) \in \SL(2, \C)$, $P=z+\lambda j \in \H^3$. The following identity holds:
\begin{equation}
      \displaystyle R(d\sigma, P)  =
 \Big(  \begin{matrix}
( d\sigma_P \, e_1) ^{ \, T} &  (d\sigma_P \, e_2)^{ \, T} &  
( d\sigma_P \,  e_3)^{ \, T} \\
\end{matrix}  \Big), 
\end{equation}
where $\Big(  \begin{matrix}
( d\sigma_P \, e_1) ^{ \, T} & ( d\sigma_P \, e_2) ^{ \, T} &  ( d\sigma_P \, e_3) ^{ \, T} \\
\end{matrix}  \Big)$
denotes the matrix whose columns are given by:
\begin{align}
 d\sigma_P \, e_1
             & \; = \;   \frac{1}{2 \Omega}
\Big[ (\overline{z}^2\overline{c}^2+z^2c^2)+2(\overline{zcd}+zcd)+(\overline{d}^2+d^2) -\lambda^2(\overline{c}^2+c^2) \Big] \,  e_1             &&  \nonumber \\
   & \; - \;   \frac{i}{2  \Omega}
\Big[ (\overline{z}^2\overline{c}^2-z^2c^2)+2(\overline{zcd}-zcd)+(\overline{d}^2-d^2) -\lambda^2(\overline{c}^2-c^2) \Big] \,  e_2 
     && \nonumber \\
             & \; - \;  \frac{\lambda}{ \Omega } \Big[ \abs{c}^2 (\overline{z}+z)  + (c\overline{d}+\overline{c}d) \Big] \, e_3,
     &&  \label{dSigmae1}
\end{align}
\begin{align}
  d\sigma_P \, e_2
             & \; = \;   \frac{i}{2  \Omega}
\Big[ (\overline{z}^2\overline{c}^2-z^2c^2)+2(\overline{zcd}-zcd)+(\overline{d}^2-d^2) +\lambda^2(\overline{c}^2-c^2) \Big] \,  e_1 
     &&  \nonumber \\
             & \; + \;   \frac{1}{2 \Omega}
\Big[ (\overline{z}^2\overline{c}^2+z^2c^2)+2(\overline{zcd}+zcd)+(\overline{d}^2+d^2) +\lambda^2(\overline{c}^2+c^2) \Big] \, e_2 
     &&  \nonumber \\
      &\; - \;  \frac{\lambda i}{  \Omega } \Big[ \abs{c}^2 (\overline{z}-z)  + (c\overline{d}-\overline{c}d) \Big] \, e_3,                       && \label{dSigmae2}
\end{align}
\begin{align}
 d\sigma_P \, e_3
             & \; = \;   \frac{\lambda}{ \Omega }
 \Big[ (\overline{z}\overline{c}^2+zc^2)+ (\overline{cd}+cd)  \Big] \,  e_1   \; - \;    \frac{\lambda i}{ \Omega } \Big[ (\overline{z}\overline{c}^2-zc^2)+ (\overline{cd}-cd)  \Big] \,  e_2                   &&  \nonumber \\
             & \; + \;  \frac{1}{ \Omega } \Big[ \abs{c}^2 (\abs{z}^2-\lambda^2) + \abs{d}^2 + (\overline{zc}d+ zc\overline{d}) \Big] \,  e_3,   &&  \label{dSigmae3}
\end{align}
where    $\Omega := \abs{ cz + d }^2 + \abs{c}^2 \lambda^2$. In addition,  $\{e_1,e_2,e_3 \} $ is a Euclidean orthonormal basis of $T_P \R^3$.
\end{lemm}
\begin{proof}
Direct calculations.
\end{proof}

\bigskip
One has the double covering epimorphism  $\Phi : \SU(2) \longrightarrow \SO(3)$ denoted by  $ \Phi ( A ) = \Phi_A$, with $\text{kernel}\,\Phi=\left\{I,-I \right\}$. This is of course the spin cover of $\SO(3)$. Explicitly, if $ A = \left(  \begin{matrix}
\alpha & \beta \\
-\overline{\beta} & \overline{\alpha} \\
\end{matrix}  \right) \in \SU(2)$, hence
\begin{equation}\label{AuxM1}
   \Phi_A = \left(  \; \begin{matrix}
  \Re \, (\alpha^2 - \beta^2)                & - \Im \, (\alpha^2 + \beta^2)          & 2 \,  \Re \, ( \alpha \beta ) \\
  \Im \, ( \alpha^2 -  \beta^2 )             & \Re \, ( \alpha^2 +  \beta^2 )             &  2 \, \Im \, (  \alpha \beta  ) \\
 - 2 \, \Re \, (  \overline{\alpha} \beta )  & 2 \, \Im \, (  \alpha \overline{\beta }) &   \abs{\alpha}^2 - \abs{\beta}^2  \\
\end{matrix} \;  \right).
\end{equation}

\begin{lemm}\label{Series9}
Let $\sigma=n[z']a[\lambda']  K = \left(  \begin{matrix}
a & b \\
c & d \\
\end{matrix}  \right) \in \SL(2, \C)$,  $g=n[z] a[\lambda]$ and $P=z + \lambda j \in \H^3$. Then the following identity holds:
\begin{equation*}
      \displaystyle  \Phi_{ T(\sigma  g )}  \; = \;  B \circ R(d\sigma, P) \circ B
\end{equation*}
where
\[
B=\left(  \begin{matrix}
1 & 0 & 0 \\
0 & 1 & 0 \\
0 & 0 & -1 \\
\end{matrix}  \right).
\]
\end{lemm}

\begin{proof}
One has
\begin{align}
\sigma    & \; = \;  \left(  \begin{matrix}
1 & z' \\
0 & 1  \\
\end{matrix}  \right)
\left(  \begin{matrix}
\sqrt{\lambda'} & 0 \\
0 & \frac{1}{\sqrt{\lambda'}}  \\
\end{matrix}  \right)
\left(  \begin{matrix}
\alpha & \beta \\
-\overline{\beta}  & \overline{\alpha}  \\
\end{matrix}  \right)
             \; = \;
\displaystyle \left(  \begin{matrix}
\sqrt{\lambda'} \, \alpha - \frac{z'}{\sqrt{\lambda'}} \, \overline{\beta}  & \sqrt{\lambda'} \, \beta + \frac{z'}{\sqrt{\lambda'}} \, \overline{\alpha} \\
- \frac{1}{\sqrt{\lambda'}} \, \overline{\beta}  &  \frac{1}{\sqrt{\lambda'}} \, \overline{\alpha}  \\
\end{matrix}  \right).           &&  \label{Ser6}.
\end{align}
In addition,
\begin{equation}\label{Ser7}
  g =
\left(  \begin{matrix}
1 & z \\
0 & 1  \\
\end{matrix}  \right)
\left(  \begin{matrix}
\sqrt{\lambda} & 0 \\
0 & \frac{1}{\sqrt{\lambda}}  \\
\end{matrix}  \right) =
\left(  \begin{matrix}
\sqrt{\lambda} & \frac{z}{\sqrt{\lambda}} \\
0 & \frac{1}{\sqrt{\lambda}}  \\
\end{matrix}  \right).
\end{equation}

From (\ref{Ser6}) and (\ref{Ser7}) it follows that:
\begin{align}
\sigma g   & \; = \;
\displaystyle \left(  \begin{matrix}
\sqrt{\lambda \lambda'} \, \alpha - \frac{\sqrt{\lambda}}{\sqrt{\lambda'}} \, z' \overline{\beta}
&    \frac{\sqrt{\lambda'}}{\sqrt{\lambda}} \, z \alpha -  \frac{zz'}{\sqrt{\lambda \lambda'}} \, \overline{\beta}
     +  \frac{\sqrt{\lambda'}}{\sqrt{\lambda}} \, \beta  + \frac{z'}{\sqrt{\lambda \lambda'}}  \, \overline{\alpha}   \\
- \frac{\sqrt{\lambda}}{\sqrt{\lambda'}} \, \overline{\beta}
&  -\frac{z}{\sqrt{\lambda \lambda'}} \, \overline{\beta} +  \frac{1}{\sqrt{\lambda \lambda'}} \, \overline{\alpha}  \\
\end{matrix}  \right) \; := \;
\left(  \begin{matrix}
a' & b' \\
c' & d' \\
\end{matrix}  \right).        &&  \label{Ser8}
\end{align}

\bigskip
We want to find the matrix $T(\sigma g)$ (\emph{i.e.,} the $\SU(2)$ part of the Iwasawa decomposition of $\sigma g$). First let us remark that $$\abs{c'}^2 + \abs{d'}^2 = \tfrac{ \lambda^2 \, \left\vert \beta  \right\vert^2 + \abs{z}^2 \, \abs{\beta}^2 + \abs{\alpha}^2 - \big( z \alpha \overline{\beta} + \overline{z} \overline{\alpha} \beta \big) } {\lambda \lambda'}.$$
Let $r:= \lambda^2 \, \left\vert \beta  \right\vert^2 + \abs{z}^2 \, \abs{\beta}^2 + \abs{\alpha}^2
              - \big( z \alpha \overline{\beta} + \overline{z} \overline{\alpha} \beta \big)$, therefore
\begin{equation}\label{Ser11}
  \sqrt{\abs{c'}^2 + \abs{d'}^2 } = \frac{\sqrt{r}}{\sqrt{\lambda \lambda'}}.
\end{equation}
From (\ref{Ser8}) and (\ref{Ser11}) we conclude:
\[
    \frac{c'}{\sqrt{\abs{c'}^2 + \abs{d'}^2 }} =  - \frac{\lambda \overline{\beta}}{\sqrt{r}} \; \hspace{3cm} \;  \frac{d'}{\sqrt{\abs{c'}^2 + \abs{d'}^2 }} =  \frac{-z \overline{\beta}+ \overline{\alpha}}{\sqrt{r}}.
\]

\bigskip
Using Iwasawa  decomposition  (lemma (\ref{DIwasawa})) we obtain:
\begin{equation}\label{Ser13}
  T( \sigma g ) =
 \left(  \begin{matrix}
\tfrac{-\overline{z}\beta+\alpha }{\sqrt{r}} & \tfrac{ \lambda \beta }{\sqrt{r}} \\
\tfrac{ -\lambda \bar{\beta} }{\sqrt{r}} & \tfrac{-z \vec{\beta}+\bar{\alpha} }{\sqrt{r}}  \\
\end{matrix}  \right) \; := \;
 \left(  \begin{matrix}
\mu  & \eta \\
-\bar{\eta} & \bar{\mu} \\
\end{matrix}  \right).
\end{equation}

The matrix $\Phi_{T( \sigma g )}$ in  (\ref{AuxM1}) is given by:
\begin{equation}\label{MarV}
  \Phi_{T( \sigma g )} =  \left(  \; \begin{matrix}
  \Re \, (\mu^2 - \eta^2)                & - \Im \, (\mu^2 + \eta^2)          & 2 \,  \Re \, ( \mu \eta ) \\
  \Im \, ( \mu^2 -  \eta^2 )             & \Re \, ( \mu^2 +  \eta^2 )             &  2 \, \Im \, (  \mu \eta  ) \\
 - 2 \, \Re \, (  \overline{\mu} \eta )  & 2 \, \Im \, (  \mu\overline{\eta }) &   \abs{\mu}^2 - \abs{\eta}^2  \\
\end{matrix} \;  \right).
\end{equation}

\bigskip

On the other hand using the Iwasawa decomposition of $\sigma$ we obtain:
\[
T(\sigma) = \left(  \begin{matrix}
\alpha & \beta \\
-\bar{\beta} & \bar{\alpha} \\
\end{matrix}  \right)
\]
with
\begin{equation}\label{T}
  \alpha = \frac{ \overline{d} }{ \sqrt{t} }  \hspace{4cm}
     \beta = -\frac{  \overline{c} }{ \sqrt{t} },
\end{equation}
where $t:= \abs{c}^2+\abs{d}^2$. Using the identities in (\ref{T}) we see that:
\begin{equation}\label{Ser15.5}
  rt = \left\vert cz+d \right\vert^2 +    \lambda^2  \abs{c}^2= \Omega.
\end{equation}

\bigskip
Let us see how $\mu$ and $\eta$ of matrix (\ref{Ser13}) change using the identities  (\ref{T}),  we obtain the following equations:
\begin{equation}\label{XX}
  \mu =  \frac{ \overline{z} \overline{c} + \overline{d}  }{\sqrt{\Omega}}  \hspace{4cm}
     \eta = - \frac{ \lambda \overline{c} }{\sqrt{\Omega}}.
\end{equation}

\bigskip
Let us compute some entries of matrix $\Phi_{T( \sigma g )}$ in (\ref{MarV}) using  identities  (\ref{XX}). We see that:
\[
   \mu \overline{\eta} =  - \frac{\lambda}{\Omega} \big(  \overline{z} \abs{c}^2 + c\overline{d} \big).
\]
Hence,
\begin{align}
\hspace{3cm} -2 \Re \, (\overline{\mu} \eta) &  \; = \;   \frac{ \lambda  }{\Omega }      \Big[ \abs{c}^2(\overline{z}+z)+(c\overline{d} + \overline{c} d)  \Big].    &&    \label{Ser21}
\end{align}

\bigskip
Consider the following difference,
\[
   \mu^2- \eta^2 =  \frac{  \overline{z}^2 \overline{c}^2 + 2 \overline{zcd} + \overline{d}^2 - \lambda^2 \overline{c}^2}{\Omega}.
\]
Hence
\begin{align}
 \Re \, (\mu^2- \eta^2) & \; = \;   \frac{1}{2\Omega}
              \Big[ (\overline{z}^2 \overline{c}^2+z^2c^2)+2(\overline{zcd}+zcd) + (\overline{d}^2+d^2) -
              \lambda^2 (\overline{c}^2+c^2) \Big].    &&             \label{Ser25}
\end{align}
\hspace{-0.5cm}
\begin{align}
 - \Im \, (\mu^2 - \eta^2) & \; = \;  \frac{i}{2\Omega}
              \Big[ (\overline{z}^2 \overline{c}^2-z^2c^2)+2(\overline{zcd}-zcd) + (\overline{d}^2-d^2) -
              \lambda^2 (\overline{c}^2-c^2) \Big].    &&             \label{Ser26}
\end{align}

\bigskip
The first column of matrix $\Phi_{T( \sigma g )}$ in (\ref{MarV}) is given by identities (\ref{Ser25}), (\ref{Ser26}) and (\ref{Ser21}),
\begin{align}
  \Phi_{T( \sigma g )} (  e_1 )^T & \; = \;
\left(  \begin{matrix}
\frac{1}{2\Omega}       \Big[ (\overline{z}^2 \overline{c}^2+z^2c^2)+2(\overline{zcd}+zcd) + (\overline{d}^2+d^2) -
              \lambda^2 (\overline{c}^2+c^2) \Big]  \\
- \frac{i}{2\Omega}          \Big[ (\overline{z}^2 \overline{c}^2-z^2c^2)+2(\overline{zcd}-zcd) + (\overline{d}^2-d^2) -
              \lambda^2 (\overline{c}^2-c^2) \Big] \\
\frac{ \lambda  }{\Omega }
                    \Big[ \abs{c}^2(\overline{z}+z)+(c\overline{d} + \overline{c} d)  \Big]  \\
\end{matrix}  \right)    &&  \nonumber \\
                                    & \; = \;
\left(  \begin{matrix}
\langle  d \sigma_P \, e_1, e_1 \rangle^{\R^3}  \\
\langle  d \sigma_P \, e_1, e_2 \rangle^{\R^3}  \\
-\langle d \sigma_P \, e_1, e_3 \rangle^{\R^3} \\
\end{matrix}  \right)    && \text {by lemma (\ref{Series5})} \nonumber \\
                                     &  \; = \;
B \circ  R(d\sigma,P) \, \big( e_1 \big)^T \; = \; 
 B \circ R(d\sigma, P) \circ  B \, \big( e_1 \big)^T.  \nonumber
\end{align}

Analogously for $e_2$ and $e_3$.
\end{proof}

\begin{Corollary-non} By the previous lemma,  $ \text{det} \, \big(  R(d\sigma,P) \big)=1$, and  $R(d\sigma, P)$ preserves the Euclidean norm. Therefore $R(d\sigma,P) \in \SO(3)$.
\end{Corollary-non}

Later we will use the following property of Wigner matrices. Let $R,T \in \SO(3)$, $l \in \N$, $b,m \in \Z$ such that $b,m \in [-l,l]$. Then
\begin{equation}\label{WI10}
    D^l_{bm} (R \circ T) \; = \; \sum_{a=-l}^{l}  D^l_{am} ( T)  \cdot  D^l_{ba} (R ).
\end{equation}
This is a consequence of the fact that for fixed $l$, the entries $ D^l_{bm}$ are the coefficients of a representation.

\subsection{Definition and properties of Eisenstein series $ \widehat{E}_{km}^l(g,s)$}

\begin{defi}\label{def0}
Let $l \in \N$, $k,m \in \Z$ such that $k,m \in [-l,l]$. We denote by $\Gamma'_{\infty} \backslash \Gamma$ the set of right cosets of $\Gamma'_{\infty}$ in $\Gamma$. If $s \in \C$ is such that $\Re(s) > 1$, the Eisenstein series $\widehat{E}_{km}^l(\cdot,s): \SL(2,\C) \longrightarrow \C$, associated to $\Gamma$ at the cusp $\infty$, are defined by the following  formula:
\[ 
    \displaystyle \widehat{E}_{km}^l (g,s) \; := \;  \cfrac{1}{[\Gamma_{\infty} : \Gamma'_{\infty} ]} \,
    ‎‎\sum_{\sigma \in  \Gamma'_{\infty} \backslash \Gamma}
        \overline{ D_{km}^l \big( \Phi_{ T(\sigma  g )^{-1} }    \big) } \; \Im \,   \sigma g (j)^{1+s},
\]
the notation  $\Im \,   \sigma g (j)^{1+s}$ corresponds to $ \big( \Im \,   \sigma g (j) \big)^{1+s}$.
\end{defi}

\bigskip
If $f_{km}^l(\cdot, s) : \SL(2,\C) \longrightarrow \C $ is defined by the formula:
\begin{equation}\label{Deffff}
   f_{km}^l(g,s) \; := \; \overline{ D_{km}^l \big( \Phi_{ T( g )^{-1}} \big) } \; \Im \,  g (j)^{1+s}, \; \; \; \Re(s) > 1.
\end{equation}

Then
\begin{equation}\label{def00}
\widehat{E}_{km}^l (g,s) \; = \;  \cfrac{1}{[\Gamma_{\infty} : \Gamma'_{\infty} ]} \,
                \sum_{\sigma \in  \Gamma'_{\infty} \backslash \Gamma} f_{km}^l (\sigma g,s ).
\end{equation}

\bigskip
\begin{Lemma-non}
The series $\widehat{E}_{km}^l (g,s)$ are well defined.
\end{Lemma-non}
\begin{proof}
Let $\gamma_\infty g \in \Gamma'_{\infty} \backslash \Gamma$, with $g=n[z]a[\lambda] K  \in \Gamma $,  $\gamma_\infty :=
  \left(  \begin{matrix}
1 & z \\
0 & 1 \\
\end{matrix}  \right) \in \Gamma'_{\infty}$ for some $z \in  \mathcal{O}_D$. Then,  $\gamma_\infty g \; = \; 
n[z+q]a[\lambda] K$ and 
\begin{equation}\label{def001}
  T (\gamma_\infty g ) \; = \; K .
\end{equation}

\bigskip
One has the following chain of equalities:
\begin{align}
f_{km}^l (\gamma_\infty g)  & \; = \;
      \overline{ D_{km}^l \big( \Phi_{ T(\gamma_\infty  g )^{-1}} \big) } \; \Im \,   \gamma_\infty  g (j)^{1+s}
                     && \text {by  (\ref{Deffff}) } \nonumber \\
             & \; = \;
      \overline{ D_{km}^l \big( \Phi_{  K^{-1} } \big) } \; \Im \,   \gamma_\infty g (j)^{1+s}
                     && \text {by (\ref{def001})} \nonumber \\
                  & \; = \;
      \overline{ D_{km}^l \big( \Phi_{ T( g )^{-1}} \big) } \; \Im \,    g (j)^{1+s}    \; = \; f_{km}^l (g).                              &&  \nonumber
\end{align}
Hence, the function $f_{km}^l$ is constant on right cosets $\Gamma'_{\infty} \backslash \Gamma$, using (\ref{def00}) we conclude that $\widehat{E}_{km}^l (g,s) $ are well defined.
\end{proof}

\bigskip
The series $\widehat{E}_{km}^l (g,s) $ converge for $\Re(s) > 1$. In fact, convergence follows from the convergence of the classic Eisenstein series ${E} (P,s) $, see   \cite{Lok04}, page 30  or 
\cite{Els98}, chapter 3.

\bigskip
\begin{lemm}\label{def003}
If $s \in \C$ is such that $\Re(s) > 1$. Then:
\[
    \widehat{E}_{km}^l (\gamma g,s) \; = \; \widehat{E}_{km}^l (g,s), \; \; \; \forall \gamma \in \Gamma.
\]
\end{lemm}
\begin{proof}
 By definition  (\ref{def0}):
\[ 
    \widehat{E}_{km}^l (\gamma g,s) \; = \;  \cfrac{1}{[\Gamma_{\infty} : \Gamma'_{\infty} ]} \,  ‎\sum_{\sigma \in  \Gamma'_{\infty} \backslash \Gamma}
        \overline{ D_{km}^l \big( \Phi_{ T(\sigma \gamma g )^{-1}} \big)  } \; \Im \,   \sigma \gamma g (j)^{1+s}.
\]
Making the change of variable $\eta = \sigma \gamma \in \Gamma'_{\infty} \backslash \Gamma$, we can re-write the previous expression as follows:
\[ 
    \widehat{E}_{km}^l (\gamma g,s) \; = \;   \cfrac{1}{[\Gamma_{\infty} : \Gamma'_{\infty} ]} \,
    ‎\sum_{\eta \in  \Gamma'_{\infty} \backslash \Gamma}
        \overline{ D_{km}^l \big( \Phi_{ T(\eta  g )^{-1}} \big)  } \; \Im \,   \eta g (j)^{1+s},
\]
The previous sum is precisely $\widehat{E}_{km}^l (g,s)$.
\end{proof}


\bigskip
\subsection{Definition and properties of the series $ \, E_{km}^l(P,s)$.}
\begin{defi}\label{def004}
Let $l \in \N$, $k,m \in \Z$ such that $k,m \in [-l,l]$, $s \in \C$ and $\Re(s) > 1$, we define the series
$
 E_{km}^l(\cdot,s): \H^3 \longrightarrow \C$, associated to $\Gamma$ at the cusp
 $\infty$,  by the formula:
\[ 
    \displaystyle E_{km}^l (P,s) \; := \; \cfrac{1}{[\Gamma_{\infty} : \Gamma'_{\infty} ]} \,
    ‎‎\sum_{\sigma \in  \Gamma'_{\infty} \backslash \Gamma}
        \overline{ D_{km}^l \big( \Phi_{ T( \sigma n[z] a[\lambda] )^{-1}} \big) } \; \Im \,   \sigma P^{1+s}
\]
with  $P=z+\lambda j  \in \H^3$. The notation $ \Im \,   \sigma P^{1+s}$ correspond to $ \big( \Im \,   \sigma P \big)^{1+s}.$ 
 \end{defi}

If  $g_P=n[z]a[\lambda]$,  $\Re(s) > 1$, then using  (\ref{Pre6}) one obtains
\begin{equation}\label{def0004}
   \widehat{E}_{km}^l (g_P,s) \; = \; E_{km}^l (P,s).
\end{equation}

\bigskip
The identity (\ref{def0004}) implies that the series $E_{km}^l (P,s)$ are well defined and are convergent in the half-plane  $\Re(s) > 1$. Although  $E_{km}^l (P,s)$  do not define in general  Eisenstein series in $\H^3$, except for $l=k=m=0$, since they are not invariant under $\Gamma$, however they do admit useful Fourier expansions.

\begin{lemm}\label{def005}
 Let $s \in \C$ with $\Re(s) > 1$, then $\forall z \in \Lambda_D \subset \C$, $\forall q \in \C$ the following equalities holds:
\[
    E_{km}^l (P_{z+q},s) \; = \; E_{km}^l (P_q,s).
\]
\end{lemm}
\begin{proof}
  Let $\gamma_\infty :=
\left(  \begin{matrix}
1 & z \\
0 & 1 \\
\end{matrix}  \right)  \in \Gamma'_{\infty}$ with  $z \in \Lambda_D$, and   $g_w=n[w]a[\lambda]$, $\forall w \in \C$. Therefore
\begin{align}
E_{km}^l (P_q,s)  & \; = \;  \widehat{E}_{km}^l (g_q,s)   && \text {by (\ref{def0004})} \nonumber \\
             & \; = \;   \widehat{E}_{km}^l ( \gamma_\infty \cdot g_q,s)   && \text {by  lemma (\ref{def003})} \nonumber \\
             & \; = \; \widehat{E}_{km}^l (g_{z+q},s) \; = \; E_{km}^l (P_{z+q},s).   && \text {by (\ref{def0004})} \nonumber
\end{align}
\end{proof}

Lemma (\ref{def005}) implies that the functions $ E_{km}^l (z+\lambda j,s)$ with $s$ and $\lambda$ fixed are periodic in $z$ with respect to the lattice $\Lambda_D$, hence they admit a Fourier expansion in the variable $z$, \emph{i.e.}
 \begin{equation*}
 \displaystyle E_{km}^l (z+\lambda j,s) \; = \; \sum_{w \in \Lambda_D^\ast}  b_{km}^l (\lambda,s)_w \,  e^{  2 \pi i \, \langle z, w \rangle },
\end{equation*}
where $\Lambda_D^\ast= \{ w \in \C \, ; \,  \langle  w, q \rangle \in \Z, \; \; \forall q \in \Lambda \}$ is the dual lattice of $\Lambda_D$. The Fourier coefficients are given as follows:

 \begin{equation}\label{def9}
         \displaystyle  b_{km}^l (\lambda,s)_w \; = \; \frac{1}{\abs{\Lambda_D}}
         \int_{\R^2 / \Lambda_D}  E_{km}^l (z+\lambda j,s)  \, e^{ - 2 \pi i \, \langle z, w \rangle } \, dxdy,
\end{equation}
with $w \in \Lambda_D^\ast$. Here $\abs{\Lambda_D}$ denotes the area  of the flat torus $\R^2 / \Lambda_D$.

\bigskip 
\subsection{Relation between the series  $E_{km}^l (P,s)$ and the auxiliary series $ H_{km}^l (P,s)$}

If $s \in \C$ with $\Re \, (s) > 1$, $g_P=n[z]a[\lambda]$ with $P=z+\lambda j \in \H^3$, $\sigma = \left(  \begin{matrix}
a & b \\
c & d \\
\end{matrix}  \right) \in \SL(2, \C)$, by lemma (\ref{Series9}) it holds
\[ 
  \Phi_{ T(\sigma  g )}  \; = \;  B \circ R(d\sigma,P) \circ B.
\] 
Then
\begin{equation}\label{Relac2}
R(d\sigma,P)^{-1} \; = \; (-B) \circ  \Phi_{ T(\sigma  g )^{-1} }  \circ  (- B),
\end{equation}
where $-B \in \SO(3)$.

\bigskip
Let $l \in \N$, $k,m \in  \Z$, such that $k,m \in [-l,l]$. Hence, by (\ref{WI10})
\begin{align}
D_{km}^l \big( R(d\sigma,P)^{-1} \big)   & \; = \;  \sum_{a=-l}^{l} \sum_{c=-l}^{l}   D^l_{cm} (-B) \cdot D^l_{ac}  \big( \Phi_{ T(\sigma  g )^{-1} } \big) \cdot  D^l_{ka} (-B).
                     &&  \label{Relac2.5}
\end{align}

\bigskip
To prove the equidistribution of the horocycle foliation we need the following series:
\begin{equation}\label{Hlkm}
  H_{km}^l (P,s)  \; := \;   \cfrac{1}{[\Gamma_{\infty} : \Gamma'_{\infty} ]} \,
 \sum_{\sigma \in  \Gamma'_{\infty} \backslash \Gamma} \overline{ D_{km}^l \big( R(d\sigma,P)^{-1}  \big) }  \,   \Im \, \sigma P^{1+s}.
\end{equation}

Substituting (\ref{Relac2.5}) in  (\ref{Hlkm}) one has:
\begin{align}
 H_{km}^l (P,s)  & \; = \;    \sum_{a=-l}^{l} \sum_{c=-l}^{l}
                \overline{D^l_{cm} (-B)} \cdot \overline{ D^l_{ka} (-B)  } \,
                  \Bigg[  \cfrac{1}{[\Gamma_{\infty} : \Gamma'_{\infty} ]} \,
                  ‎‎\sum_{\sigma \in \Gamma'_{\infty} \backslash \Gamma}  
                   \overline {D^l_{ac} \big ( \Phi_{ T(\sigma  g )^{-1} } \big) } \;
       \Im \, \sigma P^{1+s} \Bigg]   &&  \nonumber \\
             & \; = \;   \sum_{a=-l}^{l} \sum_{c=-l}^{l}
                \overline{D^l_{cm} (-B)} \cdot \overline{ D^l_{ka} (-B)  } \; E_{ac}^l (P,s).   &&  \nonumber
\end{align}

Since $-B = ROT(\pi,0,0)$ one has that $D^l_{ab} (-B)=  e^{ia \pi} \, \delta_{ab}$. Therefore it has
\begin{equation}\label{relacionHE}
   H_{km}^l (P,s)    =  e^{-i (k+m) \pi}  \cdot E_{km}^l (P,s). 
\end{equation}

\bigskip
\subsection{Meromorphic continuation}
In \cite{Langlands} Robert Langlands developed the general theory of Eisenstein series, proved that they satisfy a functional equation and proved the analytic or meromorphic continuation of 
$\widehat{E}_{km}^l (g,s)$ with respect $s$ to all the complex plane. This theory is enhanced in \cite{MW},  chapter IV, where a proof of meromorphic continuation of Eisenstein series is given and attributed to Herv\'e Jacquet. See also \cite{Garrett} on the web page of Paul Garrett.
Also, since the Bianchi groups are special, the analytic or meromorphic continuation can be obtained without using the Langlands work, see \cite{BruMoto}, page 44, for a proof in the particular case $\Gamma_1=\SL(2, \Z[i])$, the general case is similar.

\bigskip
In particular, by  (\ref{def0004}) and \cite{Langlands, MW}, one has the meromorphic continuation of the series
$ E_{km}^l (P,s)$. \emph{More specifically we will assume the following result}:

\begin{theo}\label{conje}
Let $l \in \N$, $k,m \in \Z$ such that $k,m \in [-l,l]$, $s \in \C$  with $\Re(s) > 1$. The  series  $\widehat{E}_{km}^l (g,s)$ admits a meromorphic continuation in the variable $s$ to all the complex plane. The same thing is true for the series $ E_{km}^l (P,s)$. Furthermore, if $l > 0$ then  $\widehat{E}_{km}^l (g,s)$ and $ E_{km}^l (P,s)$ admit an analytic continuation for $\Re(s) \geq 1$. If $l=k=m=0$ it is well-known that the series $\widehat{E}_{00}^0 (g,s) = E_{00}^0 (P,s)= \tfrac{1}{ [\Gamma_{\infty} : \Gamma'_{ \infty} ]} \, E(P,s)$ have a simple pole at $s=1$.
\end{theo}


\section{Geodesic flow and horocycle foliations}

Geodesic flows on Riemannian manifolds of negative sectional curvature are the paradigm of hyperbolic dynamics in the sense
of Smale \cite{Smale} and they are examples of Anosov flows \cite{Anosov}. In what follows we will restrict our attention to 
the case of hyperbolic 3-manifolds or orbifolds \ie manifolds (or orbifolds) of the form  $\H^3/\Gamma$ where
$\Gamma$ is a discrete subgroup of isometries of $\H^3$. In this case the geodesic flow and corresponding horocycle foliations
are orbits of actions of subgroups on locally homogeneous spaces. As references on this subject we recommend \cite{Da1, DB, GV, MS, Starkov}.

\bigskip
The  unit tangent bundle of $\H^3$ is:
\[
  \displaystyle  T_1 (\H^3) \; = \; \bigcup_{(w,\eta) \in \H^3} \S^{\H^3}_{(w,\eta)},
\]
where
\[
  \S^{\H^3}_{(w,\eta)} \; := \; \big \{ \, v \in T_{(w,\eta)} \H^3 \, ; \, \norm{ v }^{\H^3}_{(w,\eta)} = 1     \, \big \}.
\]

In addition,
\[
  \displaystyle  T_1(M_D)  \; = \; T_1 (\H^3)  / \Gamma_D^\ast,
\]
where $\Gamma_D^\ast$ refers to the natural action of $\Gamma_D$ on $T_1 (\H^3)$, i.e., one has the usual action of  $\Gamma_D$ on  $\H^3$  and the action on the unit tangent vector is given via the differential.

\bigskip
First we  briefly recall some facts about the geodesic flow and the associated horocycle foliations (the stable and unstable foliations) of a hyperbolic 3-manifold $N$.
The geodesic flow $g_t:T_1(N)\longrightarrow{T_1(N)},\quad t\in\R$, is defined by a nonsingular vector field $X$ on the unit tangent bundle. \emph{This flow preserves Liouville measure.} For a hyperbolic manifold this flow is of Anosov type, which means that its tangent bundle splits as a real analytic (this is true for hyperbolic manifolds of constant negative curvature) Whitney sum of three bundles
$T(T_1(N))=E^s\oplus{E^u}\oplus{E^1}$. These bundles are invariant under the differential $dg_t$ for all
$t\in\R$. The bundle $E^1$ is the line bundle tangent to the orbits of $g_t$, i.e.  the line field generated by $X$. There exists constants $C,c>0$ and such that
\[
|dg_t(v)|\leq{C}e^{-ct}|v|, \quad \forall \,t>0, \, \text{and vector}\, v \,\, \text{in the bundle}\,\, E^s.
\]
\[
|dg_t(w)|\geq{C}e^{ct}|w|, \quad \forall \,t>0, \, \text{and vector}\, w \,\, \text{in the bundle}\,\, E^u.
\]

Since $N$ is a hyperbolic manifold the distributions $E^{ss}$ and $E^{uu}$ are both smoothly integrable (in fact they are real analytically integrable)
\ie  they are tangent to smooth 2-dimensional foliations $\cF^{ss}$ and $\cF^{uu}$, respectively. This are the \emph{strongly stable and unstable foliations} of the geodesic flow, respectively. The geodesic flow preserves these foliations so that the leaves of each of these foliations are permuted by $g_t$. Therefore
$E^s\oplus{E^1}$ and $E^u\oplus{E^1}$ are also integrable and tangent to smooth foliations
 $\cF^{s}$ and $\cF^{u}$ called \emph{center stable and center unstable foliations}, respectively. The leaves are densely immersed copies of $\R^3$ or $\R^2\times\bS^1$. There are countably many leaves diffeomorphic to $\R^2\times\bS^1$ which correspond to the closed geodesics of $N$.

If $N$ is compact the leaves of both $\cF^{ss}$ and $\cF^{uu}$ are densely embedded copies of $\R^2$. If $N$ is not compact but has finite volume then it has a finite number $n$ of cusps and there are exactly $n$ leaves $L^s_1,\dots,L^s_n$ and $L^u_1,\dots,L^u_n$ of  $\cF^{s}$ or $\cF^{u}$ which are diffeomorphic to $\bT^2\times\R$. If $M$ has only one cusp then all leaves of $\cF^{ss}$ respectively $\cF^{uu}$ are diffeomorphic to $\R^2$ except for a one-parameter family of tori contained in $L^s_1$ or $L^u_1$, respectively. When $M$ is not compact but has finite volume one has Dani's dichotomy:
\emph{A leaf of  $\cF^{ss}$  (respectively $\cF^{uu}$) is either dense or a torus} \cite{Da}.

\begin{rem}When $N$ is a (complete) hyperbolic orbifold of finite volume the stable and unstable foliations are actually singular foliations and the leaves are 2-dimensional Euclidean orbifolds.
\end{rem}
\begin{rem} {\bf We will use from now on the name \bf{horocycle foliations}
for $\cF^{ss}$ and $\cF^{uu}$.}
\end{rem}

\section{Equidistribution of horocycle foliations}
Let 
\[
 C^{\infty}_{0} \big( T_1 (M_D) \big) \; := \; \{ \, f:T_1 (M_D) \longrightarrow \C \, ; \, f \text{ is differentiable with compact support}  \, \},
\]
\[
C_{0}^0 \big( T_1 (M_D) \big) \; := \; \{ \, f:T_1 (M_D) \longrightarrow \C \, ; \, f \text{ is continuous with compact support}  \, \}.
\]

\bigskip
\noindent Let $\pi:{T_1} (\H^3) \longrightarrow{T_1(M_D)}$  be the canonical projection.
\vskip.2cm
\noindent	For each $\lambda \in \R$ with $\lambda > 0$ consider the surface
\[
    \mathcal{P}_\lambda \; := \; \{ (z,\lambda, e_1) \; ; \;  z \in \C  \}  \subset T_1 (\H^3).
\]
We claim that $\pi(\mathcal{P}_\lambda)\subset{T_1(M_D)}$ is a closed Euclidean
immersed 2-orbifold. We will denote this immersed 2-orbifold by $\widetilde{T}_\lambda $. This fact follows immediately from the fact that if
$\gamma^{w,\epsilon} := \left( \begin{matrix}
\epsilon &  w \\
0 & \epsilon^{-1} \\
\end{matrix}  \right) \in \Gamma_{D,\infty}$, with $w \in \mathcal{O}_D$ and $\epsilon\in\mathcal{O}^{\times}_D $ a unit,
 then
\[
 \gamma^{w,\epsilon} (z + \lambda j) \; = \; \epsilon^2 z+\epsilon{w} + \lambda j,
\]
\[
   d \gamma^{w,\epsilon}  (e_1)_{(z,\lambda)} = \epsilon^2(e_1)_{(\epsilon^2 z+\epsilon{w},\lambda)}.
\]
\noindent The orbifold singularities arise from the units of the field.

\begin{rem} The family of tori $\left\{T_\lambda\right\}_{\lambda>0}$ is the family of compact horocycles
(stable leaves) in the previous section.
\end{rem}

\begin{defi}\label{nulambda}
Let $\lambda \in \R$ such that $\lambda > 0$. We define probability measures $\nu (\lambda)$  with support on the immersed 2-orbifold  $\widetilde{T}_\lambda \subset T_1 (M_D)$ as follows:
 \[
     \nu (\lambda ) (f) \; : = \;  \frac{1}{\abs{\Lambda_D}} \int_{\R^2 / \Lambda_D}  f ( x + iy +\lambda j, e_1 ) \, dxdy,
     \; \; \;  \forall f \in C^{\infty}_{0} \big( T_1 (M_D) \big).
 \]
\end{defi}

\bigskip
Following Zagier \cite{Zag79} and Sarnak \cite{Sar80}, one considers a ``Mellin transform'' of the measures $\nu(\lambda)$  with argument $s \in \C$ such that $\Re (s)>1$
\[
 E(s) \; : = \; \int_0^\infty \nu (\lambda) \, \lambda^{s-2}   d\lambda.
\]

\begin{defi}\label{Mellinf}
If $s \in \C$ is such that $\Re(s) >1$ and $f \in  C^{\infty}_{0} \big( T_1 (M_D) \big)$, the \emph{Mellin transform}  $\mathfrak{M} (f,s)$ of $f$ is defined by the following equation:
\begin{align}
\mathfrak{M}(f,s) \; = \; \langle E(s), f \rangle   & \; := \;  \int_0^\infty \nu (\lambda)(f) \, \lambda^{s-2}  d\lambda  &&  \nonumber \\
             & \; = \;  \frac{1}{\abs{\Lambda_D}}  \int_0^\infty \int_{\R^2 / \Lambda_D}  f ( x+yi + \lambda j, e_1 ) \, \lambda^{s+1}  \, \frac{dxdyd\lambda}{\lambda^3}.   &&  \nonumber
\end{align}
\end{defi}

\bigskip
Let $(w,\eta) \in \H^3$ and $\{e_1,e_2,e_3\}$ the Euclidean basis in $ T_{(w,\eta)} \R^3$.  We will denote by $\vec{v}(\vartheta, \varphi)_{(w,\eta)}$ the vector in  $ T_{(w,\eta)} \R^3$ defined as follows
\begin{equation}\label{LP1}
\vec{v}(\vartheta, \varphi)_{(w,\eta)} := \eta \, \cos \varphi \,   \sin \vartheta \cdot e_1
 + \eta \,  \sin \varphi \, \sin \vartheta \cdot e_2
 + \eta \, \cos \vartheta \cdot e_3.
\end{equation}

The following equality hold:
\begin{equation}\label{LP3}
 \norm{\vec{v}(\vartheta, \varphi)_{(w,\eta)} }^{\R^3}     \; = \;  \eta, \hspace{2cm} 
  \norm{\vec{v}(\vartheta, \varphi) }^{\H^3}_{(w,\eta)}    \; = \; 1.
\end{equation}

\bigskip
Let $f \in C^\infty_0 \big( T_1(M_D) \big)$, equivalently, $f$   is given by a differentiable function with compact support $f:T_1 (\H^3) \longrightarrow \C$  (we use the same letter) which is $\Gamma^\ast_D$ invariant. Let $P=z+\lambda j \in \H^3$, by (\ref{LP3}) we know that $\vec{v}(\vartheta, \varphi)_P  \in \S^{\H^3}_P$, hence we can evaluate $f$ in $(P, \vec{v}(\vartheta, \varphi)_P ) \in T_1 (\H^3)$, and invariance implies:
\begin{equation}\label{LP5}
 f \big( P, \vec{v}(\vartheta, \varphi)_P \big) \; = \;
    f \big( \gamma P, d\gamma_P \, \vec{v}(\vartheta, \varphi)_P \big), \; \; \; \forall \gamma \in \Gamma.
\end{equation}

\bigskip
Restricting  $f$ to the fibers $ \S^{\H^3}_P  \underset{\text{difeo.}}{\cong} \S^2 \longrightarrow \C$ we obtain the following function:
\[
   f \big( P,  \vec{v}( \cdot , \cdot )_P \big) : \S^2 \longrightarrow \C.
\]

\bigskip
Assuming that $f$ admits a Laplace series in each fibre one has,
\begin{equation}\label{LP7}
  f \big( P,  \vec{v}( \vartheta, \varphi )_P \big) \; = \;
   \sum_{l=0}^\infty  \sum_{m=-l}^l  \widehat{f}_m^{\, l} (P) \cdot   Y^l_m (\vartheta, \varphi),
\end{equation}
where $(\vartheta, \varphi)$ denotes the point in $\S^2 \subset \R^3$ with spherical coordinates $(1, \vartheta, \varphi)$, one has also,
\begin{equation}\label{LP8}
  \widehat{f}_m^{\, l}  (P)  \; = \;
      \int_{\S^2} f \big( P, \vec{v}( \vartheta, \varphi )_P   \big)
       \, \overline{ Y_m^l (\vartheta, \varphi) } \, \sin \vartheta \, d \vartheta  d \varphi.
\end{equation}

\bigskip
Setting $\vartheta = \tfrac{\pi}{2}, \varphi = 0$ in (\ref{LP7}) since $\vec{v}( \tfrac{\pi}{2}, 0 )_P = \lambda \cdot {e_1}$, one write:
\begin{equation}\label{LP10}
  f ( P,  e_1 ) \; = \; \sum_{l=0}^\infty  \sum_{m=-l}^l  \widehat{f}_m^{\, l} (P) \cdot   Y^l_m ( e_1 ).
\end{equation}

\begin{lemm}\label{LP}
Let $l \in \N$, $m \in \Z$ be such that $m \in [-l,l]$, $f \in C^\infty_0 \big(T_1(M_D)\big)$. Then
\[
       \widehat{f}_m^{\, l} \big( \gamma P \big) \; = \; \sum_{k=-l}^l \overline{ D_{km}^l \big( R(d\gamma,P)^{-1}  \big) } \cdot
       \widehat{f}_k^{\, l} (P), \; \; \; \forall \gamma \in \Gamma.
\]
\end{lemm}
\begin{proof}
From (\ref{LP8}) one has:
\begin{equation}\label{LP11}
  \widehat{f}_m^{\, l} \big( \gamma P \big) \; = \;
      \int_{\S^2} f \big( \gamma P, \, \vec{v}( \vartheta, \varphi )_{\gamma P }   \big)
       \; \overline{ Y_m^l (\vartheta, \varphi) } \, \sin \vartheta \, d \vartheta \, d \varphi.
\end{equation}

\bigskip
Suppose we change variables in $\S^2 \subset \R^3$ using the rotation $R(d\gamma,P) \in \SO(3)$ defined  in (\ref{Rota9}),
\begin{equation}\label{LP13}
  R(d\gamma,P) \Bigg[ \frac{1}{\lambda} \; \vec{v}( \vartheta', \varphi' )_P \Bigg] \; = \;
    \frac{1}{\Im \,  \gamma P }  \; \vec{v}(\vartheta, \varphi)_{\gamma P}.
\end{equation}

\bigskip
The spherical harmonics change as follows under this rotation
\begin{align}
 Y^l_{m} (\vartheta, \varphi)  & \; = \; Y^l_{m} \Bigg( \frac{1}{\Im \, \gamma P}  \; \vec{v}(\vartheta, \varphi)_{\gamma P} \Bigg)
                              && \nonumber \\
     & \; = \;  Y^l_{m} \Bigg(  R(d\gamma,P) \Bigg[ \frac{1}{\lambda} \; \vec{v}( \vartheta', \varphi' )_P \Bigg] \Bigg)       && \text {by (\ref{LP13})} \nonumber \\
            & \; = \;  \sum_{k=-l}^{l} D^l_{km} \big( R(d\gamma,P)^{-1}  \big) \cdot
             Y^l_{k} \bigg( \frac{1}{\lambda} \; \vec{v}( \vartheta', \varphi' )_P \bigg).
     &&  \text {by (\ref{SumatoriaWigner})} \nonumber
\end{align}
Summarizing,
\begin{equation}\label{LP14}
  Y^l_{m} (\vartheta, \varphi) \; = \;
      \sum_{k=-l}^{l} D^l_{km} \big( R(d\gamma,P)^{-1}  \big) \cdot Y^l_{k} ( \vartheta', \varphi' ).
\end{equation}

\bigskip
On the other hand,
\begin{align}
 f \big( P, \, \vec{v}(\vartheta', \varphi')_P  \big)  & \; = \;
        f \big( \gamma P, d\gamma_P \vec{v}(\vartheta', \varphi')_P \big)
                 &&  \text {by (\ref{LP5})} \nonumber \\
            & \; = \;
        f \bigg( \gamma P, \Im \, \gamma P \cdot
         R(d\gamma, P) \Big[ \tfrac{1}{\lambda} \, \vec{v}( \vartheta', \varphi' )_P \Big]     \bigg)
          &&   \text {by (\ref{Rota9})} \nonumber \\
           & \; = \;
           f \big( \gamma P, \vec{v}(\vartheta, \varphi)_{\gamma P}  \big)             &&   \text {by (\ref{LP13})} \nonumber 
 \end{align}
In short, the invariance implies:
\begin{equation}\label{LP16}
  f \big( P, \vec{v}(\vartheta', \varphi')_P  \big)   \; = \;
f \big( \gamma P, \vec{v}(\vartheta, \varphi)_{\gamma P}  \big).
\end{equation}

\bigskip
In addition since $R(d\gamma,P)$ is an isometry of $\S^2$ one has:
\begin{equation}\label{LP17}
\sin \vartheta \, d \vartheta d \varphi \; = \; \sin \vartheta' \, d \vartheta' d \varphi'.
\end{equation}

Substituting (\ref{LP14}), (\ref{LP16}) and (\ref{LP17}) in (\ref{LP11}) we obtain:
\begin{align}
\widehat{f}_m^{\, l} \big( \gamma P \big)  & \; = \;  \int_{\S^2}   f \big( P, \vec{v}(\vartheta', \varphi')_P  \big)
  \;  \overline{ \sum_{k=-l}^{l} D^l_{km} \big( R(d\gamma,P)^{-1}  \big)  \cdot    Y^l_{k} (\vartheta', \varphi') }
          \sin \vartheta' \, d \vartheta' d \varphi' && \nonumber \\
           & \; = \;
             \sum_{k=-l}^{l} \overline{ D^l_{km} \big( R(d\gamma ,P)^{-1}  \big)  }
           \cdot  \int_{\S^2}  f \big( P, \vec{v}(\vartheta', \varphi')_P  \big)
            \;  \overline{  Y^l_{k} (\vartheta', \varphi') }    \, \sin \vartheta' \, d \vartheta'  d \varphi'
                            &&   \nonumber \\
                                                  & \; = \;
             \sum_{k=-l}^{l} \overline{ D^l_{km} \big( R(d\gamma,P)^{-1}  \big)  } \cdot \widehat{f}_k^{\, l} (P).
                            &&  \text {by (\ref{LP8})} \nonumber
\end{align}
\end{proof}
Recall that
\begin{equation}\label{Anillo}
        \mathcal{O}_D  \; = \;  \left\{
        \begin{tabular}{cc}
        	$ \Z \big[ i \sqrt{D} \big] \; \; \; \; \, $                         & if $D \equiv 1,2 \text{ mod } 4 $ \\
        	$\Z \big[ \frac{1}{2} + i \frac{\sqrt{D}}{2} \big]$ &  if  $D \equiv 3 \text{ mod } 4 .$ \, \, \\
        \end{tabular}
\right.
\end{equation}

\bigskip
The ring of integers $\mathcal{O}_D$  is a lattice $\Lambda_D:=  \Z + w_D \Z  \subset \C$, where
\[
         w_D:  \; = \; \left\{
        \begin{tabular}{ccc}
        	$  i \sqrt{D}$     & \;if $D \equiv 1,2 \text{ mod } 4 $  \\
            $  \frac{1+i\sqrt{D}}{2}$ & if $D \equiv 3 \text{ mod } 4. \,  $  \\
        \end{tabular}
\right.
\]

We denote by $C_D$ the closed subset of $\C$ whose boundary is the parallelogram generated by  $1$ and $w_D$. Let
\begin{equation}\label{LS2}
  S_D \; := \; \{ \, (z,\lambda) \in \H^3 ; z \in C_D \,  \},
\end{equation}
the following identity holds:
\begin{equation}\label{LS3}
      \displaystyle S_D \; = \; \bigcup_{\sigma \in  \Gamma'_{\infty} / \Gamma} \sigma(F_D),
\end{equation}
where $F_D$ is an appropriate fundamental domain of $ \Gamma_D$.

\bigskip
\begin{prop}\label{PImp} Let  $f \in C^\infty_0 \big( T_1(M_D) \big)$, $s \in \C$ such that  $\Re(s)>1$. Then
\[
 \mathfrak{M}(f,s)  \; = \; \frac{[\Gamma_{\infty} : \Gamma'_{\infty} ]}{\abs{\Lambda_D}} \,
        \sum_{l=0}^\infty  \sum_{m=-l}^l \sum_{k=-l}^l Y_m^l ( e_1 )
    \int_{F_D}   \widehat{f}_k^{\, l} ( P ) \; H_{km}^l (P,s) \, \frac{dx dy d\lambda}{\lambda^3},
\]
where $P=z+\lambda j \in \H^3$.
\end{prop}
\begin{proof}
One has the following chain of equalities:
\begin{align}
\abs{\Lambda_D} \cdot \mathfrak{M}(f,s)  & \; = \int_0^\infty \int_{C_D}   f(P,e_1) \, \Im \, P^{1+s} \,
             \frac{dx dy d\lambda}{\lambda^3}      &&  \text {by definition (\ref{Mellinf})}  \nonumber \\
             & \; = \;  \int_0^\infty  \int_{C_D}
                              \sum_{l=0}^\infty  \sum_{m=-l}^l  \widehat{f}_m^{\, l} (P)   \,  Y_m^l ( e_1 ) \;
                               \Im \, P^{1+s} 
                               \, \frac{dx dy d\lambda}{\lambda^3}
                      && \text {by (\ref{LP10})} \nonumber \\
             & \; = \;  \sum_{l=0}^\infty  \sum_{m=-l}^l  Y_m^l (e_1 )   \int_{S_D}   \widehat{f}_m^{\, l} (P)   \,
                    \Im \, P^{1+s} \, \frac{dx dy d\lambda}{\lambda^3}      && \text {by (\ref{LS2})}  \nonumber \\
             & \; = \;  \sum_{l=0}^\infty  \sum_{m=-l}^l  Y_m^l ( e_1 )     \sum_{\sigma \in  \Gamma'_{\infty} / \Gamma}
              \int_{\sigma(F_D)}   \widehat{f}_m^{\, l} (P)   \,
             \Im \, P^{1+s} \, \cfrac{dx dy d\lambda}{\lambda^3}.      && \text {by (\ref{LS3})}  \label{LS4}
\end{align}

\bigskip
Now we will study the integral given by equation (\ref{LS4}). First consider the change of variables:
\begin{equation*}
  P \; = \; \sigma P', \; \; P'=z'+ \lambda' j,
\end{equation*}
where $\sigma \in  \Gamma'_{\infty} / \Gamma$. Since $\sigma: F_D \longrightarrow \sigma(F_D)$ is an isometry it follows that:
\begin{align}
 \int_{\sigma(F_D)}   \widehat{f}_m^{\, l} (P)   \, \Im \, P^{1+s} \, \frac{dx dy d\lambda}{\lambda^3}  & \; = \;
      \int_{F_D}   \widehat{f}_m^{\, l} \big( \sigma P' \big)  \, \Im \, \sigma P'^{1+s} \, \frac{dx' dy' d\lambda'}{\lambda'^3}.      &&  \nonumber
\end{align}

\bigskip
By lemma (\ref{LP})
\begin{equation}\label{LS8}
 \int_{\sigma(F_D)}   \widehat{f}_m^{\, l} (P)   \; \Im \, P^{1+s} \, \frac{dx dy d\lambda}{\lambda^3}  \;  = \;
       \int_{F_D}   \sum_{k=-l}^l \overline{ D_{km}^l \big( R(d\sigma,P)^{-1}  \big) } \;
       \widehat{f}_k^{\, l} (P)  \; \Im \, \sigma P^{1+s}
             \, \frac{dx dy d\lambda}{\lambda^3}.
\end{equation}

\bigskip
Substituting the identity (\ref{LS8}) in equation (\ref{LS4}) we obtain the first identity in the following chain of identities:
\begin{align}
\abs{\Lambda_D} \cdot \mathfrak{M}(f,s)   & \; = \;  \sum_{l=0}^\infty  \sum_{m=-l}^l  Y_m^l ( e_1 )
     \sum_{\sigma \in  \Gamma'_{\infty} / \Gamma}       \int_{F_D}  \sum_{k=-l}^l   \overline{ D_{km}^l \big( R(d\sigma,P)^{-1}  \big) }  \;       \widehat{f}_k^{\, l} (P)  \; \Im \, \sigma P^{1+s}  \, \frac{dx dy d\lambda}{\lambda^3}
                 &&  \nonumber \\
             & \; = \;   \sum_{l=0}^\infty  \sum_{m=-l}^l \sum_{k=-l}^l  Y_m^l ( e_1 )    \int_{F_D}   \widehat{f}_k^{\, l} (P) \;
       \Bigg[  \sum_{\sigma \in  \Gamma'_{\infty} / \Gamma} \overline{ D_{km}^l \big( R(d\sigma, P)^{-1}  \big) }  \;
       \Im \, \sigma P^{1+s}  \Bigg]
                 \frac{dx dy d\lambda}{\lambda^3}        &&  \nonumber \\
      & \; = \; \sum_{l=0}^\infty  \sum_{m=-l}^l \sum_{k=-l}^l Y_m^l ( e_1 )
                 \int_{F_D}    \widehat{f}_k^{\, l} (P) \; [\Gamma_{\infty} : \Gamma'_{\infty} ] \; H_{km}^l (P,s)
                  \, \frac{dx dy d\lambda}{\lambda^3}.
                             \hspace{0.9cm}  \text {by (\ref{Hlkm})} && \nonumber
\end{align}
\end{proof}

Using  lemma (\ref{LemasPrevios3}) we know that  $\mathfrak{M}(f,s)$ is well defined for $\Re(s) >1$,  by lemma (\ref{LemasPrevios4})
\[
      \mathfrak{M} (f,s) \; = \;
      \frac{1}{\abs{\Lambda_D}}  \int_{ M_D } f(P, e_1 ) \, E (P,s) \, \frac{dx dy d\lambda}{ \lambda^3}, \; \; \;
      \Re(s) >1.
\]
Hence, we can extend properties of the  Eisenstein series $E (P,s)$ to $\mathfrak{M} (f,s)$, in particular, $\mathfrak{M} (f,s)$  admits a meromorphic continuation to all of  $\C$ with a simple pole at $s=1$.

\bigskip
Let us consider now the function  $\phi(s)$ which appears in the zeroth coefficient of the Fourier series of the classical Eisenstein series $E(P, s)$ associated to $\Gamma_D$ at $\infty$, in other words
\begin{equation}\label{phiCoef0}
  \phi(s) \; = \; \frac{\pi}{s \, \abs{\Lambda_D} }  \cdot  \frac{\zeta_D  (s)}{\zeta_D  (s+1)},
\end{equation}
where $\zeta_D  (s)$ is the Dedekind zeta function associated to the number field $\Q(\sqrt{-D})$.

\bigskip
We know that $ \phi(s)$ admits an analytic or meromorphic continuation  to all of $\C$ (see theorem 1.2, page 232, in \cite{Els98}), and it has a finite number of simple  poles located in the interval $(0,1]$ (theorem 1.11, page 244, in \cite{Els98}). We denote this poles by $s_1,s_2, \ldots, s_p$, except for the pole at $1$. In addition one can assume that
\[
   1 > s_1 > s_2 > \cdots > s_p > 0.
\]

Let $\nu^a$ with $a \in \{1,2, \ldots ,p\}$ be the distributions on $T_1 (M_D)$ defined by $\nu^a  \; := \;  \text{Res} \big(  E(s), s=s_a   \big)$.
Equivalently, if $f \in  C^{\infty}_{0} \big( T_1 (M_D) \big)$ then
\begin{equation}\label{Defnus}
  \nu^a (f) \; := \; \text{Res} \big(  \mathfrak{M} (f,s), s=s_a   \big).
\end{equation}

\bigskip
We will denote by  $\nu$ normalized Liouville measure on $T_1 (M_D)$. That is,
\begin{equation}\label{Defnu}
  \nu (f) \; : = \; \frac{1}{4 \pi \, \text{Vol} ( M_D)}     \int_{T_1 (M_D)}
     f \big( P,  \vec{v}(\vartheta, \varphi)_P \big) \,  \frac{\sin \vartheta \,  dx  dy  d\lambda  d\vartheta  d\varphi}{\lambda^3}, \; \; \; \forall f \in  C^{\infty}_{0} \big( T_1 (M_D) \big),
\end{equation}
with $P=x+yi+\lambda j$. 

\bigskip
 Let $f \in  C^{\infty}_{0} \big( T_1 (M_D) \big)$, we will mimic an idea of Zagier which consists in applying Cauchy's residue theorem to the function
\[
     h(s) \; := \; \mathfrak{M} (f,s) \, \lambda^{1-s},
\]
in the infinite strip $\{ \, s \in \C \, ; \, 0 \leq \Re(s) \leq 3 \, \}$. In order to do this we first consider the region $\Omega_h$ with $h >0$ which is the region in $\C$ with boundary the rectangle with vertices $(0,h)$, $(0,-h)$, $(3,-h)$ and $(3,h)$. Let $\alpha_h$ denote the line segment  joining $(3,h)$ with $(0,h)$, let $\beta_h$ be the line segment joining $(0,-h)$ to $(3,-h)$.

\bigskip
The meromorphic continuation of $ \mathfrak{M} (f,s)$ implies the meromorphic continuation of $h(s)$ to all of $\C$. The poles of $h(s)$ are the same as the poles of $ \mathfrak{M} (f,s)$. By lemma (\ref{LemasPrevios4}) the poles of $ \mathfrak{M} (f,s)$ are the same as the poles of the Eisenstein series $E(P,s)$, which in turn correspond to the poles of $\phi(s)$. Summarizing, the poles of $h(s)$ are $s_1,s_2, \ldots,s_p,1 \in (0,1]$.
By the residue theorem,
\begin{equation}\label{TeoFinal4}
\int_{\Omega_h}  \mathfrak{M} (f,s) \, \lambda^{1-s} \, ds  \; = \;
     2 \pi i \, Res \big( \mathfrak{M} (f,s) \, \lambda^{1-s}, s=1 \big) \, + \,
     2 \pi i \, \sum_{a=1}^p  \, Res \big( \mathfrak{M} (f,s) \, \lambda^{1-s}, s=s_a \big).
\end{equation}

\bigskip
By the Paley-Wiener theorem,
\begin{equation}\label{TeoFinal5}
  \lim_{h \rightarrow \infty} \int_{\alpha_h}   \mathfrak{M} (f,s) \, \lambda^{1-s} \, ds \; = \;
     \lim_{h \rightarrow \infty} \int_{\beta_h}   \mathfrak{M} (f,s) \, \lambda^{1-s} \, ds \; = \; 0.
\end{equation}

\bigskip
Taking the limit as $h \rightarrow \infty$ in (\ref{TeoFinal4}) and using the identities in (\ref{TeoFinal5}) we
conclude:
\[
\int_{3-i \infty}^{3+i \infty} \mathfrak{M} (f,s) \, \lambda^{1-s} \, ds \, + \,
\int_{i \infty}^{-i \infty} \mathfrak{M} (f,s) \, \lambda^{1-s} \, ds \; = \; \hspace{7cm}
\]
\vspace{-0.2cm}
\begin{equation}\label{TeoFinal6}
  \hspace{3.2cm}   2 \pi i \, Res \big( \mathfrak{M} (f,s) \, \lambda^{1-s}, s=1 \big) \, + \,
     2 \pi i \, \sum_{a=1}^p  \, Res \big( \mathfrak{M} (f,s) \, \lambda^{1-s}, s=s_a \big).
\end{equation}

Lemma (\ref{LemasPrevios7}) implies:
\begin{equation}\label{TeoFinal7}
  \text{Res} \,  \big( \mathfrak{M} (f,s) \, \lambda^{1-s}, s=1 \big)  \; = \;   \nu (f).
\end{equation}

\bigskip
From lemma (\ref{LemasPrevios8}) we obtain the equation:
\begin{equation}\label{TeoFinal8}
  \nu (\lambda) (f)  \; = \; \frac{1}{2\pi i} \, \int_{3-i \infty}^{3+i \infty} \mathfrak{M} (f,s) \, \lambda^{1-s} \, ds.
\end{equation}

Lemma (\ref{LemasPrevios11}) implies that:
\begin{equation}\label{TeoFinal9}
  \frac{1}{2\pi i} \int_{-i \infty}^{+i \infty} \mathfrak{M} (f,s) \, \lambda^{1-s} \, ds  \; = \; o ( \lambda )  \;\; \; \; \; \text{when}\,
\lambda \rightarrow 0.
\end{equation}

Let $a \in \{0,1,\ldots, p\}$, from definition  (\ref{Defnus}) have:
\begin{equation}\label{TeoFinal10}
      \displaystyle  Res \big( \mathfrak{M} (f,s) \, \lambda^{1-s}, s= s_a \big)  \; = \;  \lambda^{1-s_a} \, \nu^a (f).
\end{equation}

\bigskip
Substituting (\ref{TeoFinal7}), (\ref{TeoFinal8}), (\ref{TeoFinal9}) and  (\ref{TeoFinal10}) in (\ref{TeoFinal6})  we obtain
\[
      \displaystyle 2 \pi i \, \nu (\lambda) (f) - 2 \pi i \, o(\lambda) \; = \;   2 \pi i  \,  \nu (f)
        +  2 \pi i \big( \lambda^{1-s_1} \, \nu^1 (f)  +  \cdots  +  \lambda^{1-s_p} \, \nu^p (f) \big).
\]

\bigskip
Summarizing, one has proved the following theorem:
\begin{tio}
Let $D \in \{ 1,2,3,7,11,19,43,67,163 \}$ and $ M_D$ be the corresponding Bianchi orbifolds.
There exist numbers $1>s_1 > s_2 > \cdots > s_p > 0$ ($p$ could be zero in which case such numbers do not appear) and corresponding distributions  $\nu^1 , \nu^2 , \cdots , \nu^p $ in $T_1 (M_D)$  such that if $f \in C^{\infty}_{0} \big( T_1 (M_D) \big)$ then:
\[
      \nu (\lambda)(f)  \; = \;   \nu (f) \, + \, \lambda^{1-s_1} \, \nu^1 (f)  +   \cdots +  \lambda^{1-s_p} \, \nu^p  (f) +
          \, o(\lambda)  \; \; \; \; \; (\lambda \longrightarrow 0),
\]
where $\nu$ is normalized Liouville measure on $T_1 (M_D)$.
\end{tio}

\bigskip
\subsection{Lemmas}
\begin{lemm}\label{LemasPrevios3}
Let $f \in C_{0}^0 \big( T_1 (M_D) \big)$, $s \in \C$,  $\Re(s) > 1$,  the function $\langle E(s), \cdot \rangle : C_{0}^0 \big( T_1 (M_D) \big) \longrightarrow  \C$ given by
\[
  \langle E(s), f \rangle \; = \; \mathfrak{M} (f,s)
\]
is a measure on $T_1 (M_D)$. 
\end{lemm}
\begin{proof}
The proof is taken from Sarnak in \cite{Sar80}, page 724. Let $f \in C_{0}^0 \big( T_1 (M_D) \big)$,  $f:T_1 (\H^3) \longrightarrow \C$ which is $\Gamma^\ast_D$ invariant, $P=z+\lambda j$.  Define
\begin{equation*}
  \norm{f}_\infty \; := \; \sup_{(P, \vec{v}(\vartheta, \varphi)_P ) \in T_1 (M_D)}
  \left\vert f  \big(P, \vec{v}(\vartheta, \varphi)_P \big)  \right\vert < \infty.
\end{equation*}

Let  $C \in \R$ be a sufficiently big positive constant such that
\begin{equation*}
  \text{support} (f)  \subset \big \{ (P,  \vec{v}(\vartheta, \varphi)_P) \, ; \, z \in \C \; , \lambda \leq C \big \} \subset \H^3.
\end{equation*}

Then, if $\eta= \Re (s) > 1$,
\begin{align}
 \left\vert \mathfrak{M} (f,s)  \right\vert   & \; \leq \;
  \frac{1}{\abs{\Lambda_D}}  \int_0^\infty \int_{\R^2 / \Lambda_D} \left\vert  f ( P, e_1 ) \, \lambda^{s+1}  \right\vert \frac{dxdyd\lambda}{\lambda^3}   && \nonumber \\
           &  \; \leq \;
  \frac{1}{\abs{\Lambda_D}}  \int_0^\infty \int_{\R^2 / \Lambda_D}  \norm{f}_\infty \cdot  \lambda^{\Re(s)+1}
                    \frac{dxdyd\lambda}{\lambda^3}          &&  \nonumber \\
             & \; = \;
  \frac{\norm{f}_\infty}{\abs{\Lambda_D}}  \int_0^\infty \lambda^{\eta+1} \Bigg[ \int_{\R^2 / \Lambda_D}  dxdy \Bigg]
                    \frac{d\lambda}{\lambda^3}    \; = \;  \norm{f}_\infty   \int_0^\infty \lambda^{\eta-2}     d\lambda        &&  \nonumber \\
             & \; \leq \;
  \norm{f}_\infty   \int_0^C \lambda^{\eta-2}     d\lambda    \; = \;  \norm{f}_\infty   \frac{C^{\eta-1}}{\eta-1}  \; < \; \infty.
       &&  \nonumber
\end{align}
\end{proof}


\begin{lemm}\label{LemasPrevios4}
Let $f \in C_{0}^{\infty} \big( T_1 (M_D) \big)$,  $s\in \C$, $\Re(s)>1$. Then \begin{equation}
       \mathfrak{M} (f,s) \; = \;
      \frac{1}{\abs{\Lambda_D}}  \int_{ M_D } f(P, e_1 ) \, E (P,s) \, \frac{dx dy d\lambda}{ \lambda^3}, \; \; \; P=z+\lambda j.
\end{equation}
\end{lemm}

\begin{proof}
One has the following identities
\begin{align}
\int_{ \H^3 / \Gamma_D    } f(P, e_1 ) \, E (P,s) \, \frac{dx dy d\lambda}{\lambda^3}  & \; = \;
\int_{ \H^3 / \Gamma    } f(P, e_1 )
\sum_{\sigma \in  \Gamma'_{\infty} \backslash \Gamma} \, \Im  \,  \sigma P^{1+s}  \, \frac{dx dy d\lambda}{\lambda^3}  &&  \nonumber \\
             & \; = \;
\int_{ \H^3 / \Gamma    } \sum_{\sigma \in  \Gamma'_{\infty} \backslash \Gamma}
f \big( \sigma P, d \sigma_P e_1 \big) \,
  \Im \,   \sigma P^{1+s}  \, \frac{dx dy d\lambda}{\lambda^3}  &&  \nonumber \\
             & \; = \;
\int_{  \H^3 / \Gamma'_{\infty}  }  f ( P, e_1 ) \,
  \lambda^{1+s}  \, \frac{dx dy d\lambda}{\lambda^3}  &&  \nonumber \\
          & \; = \;
\int_0^\infty  \Bigg[ \int_{\R^2 / \Lambda_D} f ( P, e_1 )  \, dx dy \, \Bigg]
  \lambda^{s-2}  \, d \lambda  &&  \nonumber \\
             & \; = \;
\int_0^\infty  \nu (\lambda)(f) \,  \lambda^{s-2}  \, d \lambda 
       \; = \; \abs{\Lambda_D} \cdot \mathfrak{M} (f,s). &&  \nonumber
\end{align}
\end{proof}


\begin{lemm}\label{LemasPrevios6}
Let $f \in C_{0}^{\infty} \big( T_1 (M_D) \big)$,  $s\in \C$, $s \neq \pm 1$ and $s$ which is not a pole of $E(P,s)$. The following inequality holds:
\[
       \abs {\mathfrak{M} (f,s) } \; \leq \; \frac{ 1 }{  \abs{s^2-1 }^2 \cdot \abs{\Lambda_D}  } \, \norm{ \Delta^2 f }^2 \cdot \norm{ E^T ( \cdot, s ) }^2,
\]
where $T=T(f)$ is a sufficiently large constant. The truncated Eisenstein series $ E^T ( \cdot,  s )$ (see \cite{Els98}, page 270) are defined as follows:
\[
 E^T ( P , s ) \; := \; E( P , s) - \alpha (T, z,\lambda  , s),
 \]
with $P=z+\lambda j \in \H^3$, and
 \[
    \alpha (T, z,\lambda  , s) \; := \;
    \begin{cases}
      \lambda^{1+s} + \phi(s) \lambda^{1-s}  & \text{si} \; \lambda  \geq T \\
      0  & \text{si} \;  \lambda < T.
    \end{cases}
  \]
The function $\phi(s)$ is the one that appears in (\ref{phiCoef0}).
\end{lemm}
\begin{proof}
As a consequence of  (\ref{LemasPrevios4}) we know that for all $s \in \C$ one has
\begin{equation}\label{Pre6MM1}
       \mathfrak{M} (f,s) \; = \; \frac{1}{ \abs{\Lambda_D}}
        \int_{ M_D    } f(P, e_1 ) \, E (P,s) \,  \frac{dx dy d\lambda}{ \lambda^3}.
\end{equation}

\bigskip
On the other hand, if $s$ is not a pole of $E  ( P,s  )$ the following functional equation holds:
\begin{equation}\label{Pre6MM2}
       \Delta^2  E  ( P,s  ) \; = \; (s^2-1)^2 \,    E  ( P,s  ).
\end{equation}
See theorem 1.2,  page 232, in \cite{Els98} for a proof of this fact.

\bigskip
Substituting (\ref{Pre6MM2}) in (\ref{Pre6MM1}) it follows that if $s \neq \pm 1$ and $s$ is not a pole of $E  ( P,s  )$ then
\[
   \displaystyle \mathfrak{M} (f, s) \; = \;  \frac{1}{(s^2-1)^2 \cdot  \abs{\Lambda_D}}
   \int_{ M_D   } f(P, e_1 ) \,  \Delta^2 E (P,s) \, \frac{dx dy d\lambda}{ \lambda^3},
\]
integrating by parts,
\[
   \displaystyle \mathfrak{M} (f, s) \; = \;  \frac{1}{(s^2-1)^2 \cdot  \abs{\Lambda_D}}
   \int_{ M_D } \Delta^2 f(P, e_1 ) \,   E (P,s) \, \frac{dx dy d\lambda}{ \lambda^3}.
\]

\bigskip
Let $T=T(f) > 0$ be sufficiently big such that the support $f$ is contained in the set $\big \{(P,\vec{v}(\vartheta, \varphi)_P) ; z \in \C, \; \lambda \leq T \big  \} $, since $f\big( P, \vec{v}(\vartheta, \varphi)_P  \big)=0$ if $\lambda > T$ and $E^T(P,s) = E(P,s)$ for
$\lambda \leq T$ it follows that
\begin{equation*}
  \mathfrak{M} (f,s)  \; = \;  \frac{1}{(s^2-1)^2 \cdot  \abs{\Lambda_D}} \, \int_{ M_D } \Delta^2 f(P, e_1 ) \, E^T (P,s) \,  \frac{dx dy d\lambda}{ \lambda^3}.
\end{equation*}
By Cauchy-Schwarz inequality
\[
\abs {\mathfrak{M} (f,s) } \; \leq \; \frac{1}{\abs{s^2-1 }^2 \cdot  \abs{\Lambda_D} } \; \norm{ \Delta^2 f }^2 \;   \norm{ E^T ( \cdot, s ) }^2.
\]
\end{proof}


\begin{lemm}\label{LemasPrevios7}
If $f \in C_{0}^\infty(T_1 (M_D))$, then:
\[
       \text{Res} \,  \big( \mathfrak{M} (f,s) \, \lambda^{1-s}, s=1 \big)  \; = \;  \nu (f).
\]
\end{lemm}

\begin{proof}
Let $l \in \N$, $k,m \in \Z$ such that $k,m \in [-l,l]$. From  (\ref{Hlkm}),  it follows that if  $\Re(s) > 1$ then
\begin{equation}\label{Acerca1}
  H_{km}^l (P,s)   \; = \;  e^{-i (k+m) \pi}  \cdot E_{km}^l (P,s).
\end{equation}

By theorem (\ref{conje}) one has the analytic or meromorphic continuation of the series $E_{km}^l (P,s)$, therefore using equation (\ref{Acerca1}) we conclude that the series
$H_{km}^l (P,s) $  also admits analytic or meromorphic continuation to $\C$ in the variable $s$. Hence equation (\ref{Acerca1}) is valid for all $s \in \C$. Taking residues at $s=1$:
\begin{equation}\label{Acerca2}
 \text{Res} \, \big( H_{km}^l (P,s), s=1 \big)    \; = \;   e^{-i (k+m) \pi} \cdot  \text{Res} \, \big( E_{ac}^l (P,s), s=1 \big).
\end{equation}

\bigskip
One has  $E (P,s) = [\Gamma_{\infty} : \Gamma'_{\infty} ] \cdot E_{00}^0 (P,s)$. The residue at $s=1$ is a classical result,
\begin{equation}\label{Acerca3}
        \text{Res} \, \big( E(P,s) , s=1  \big)  \; = \;   \frac{ \abs{\Lambda_D} }{ \text{Vol}(M_D) }.
\end{equation}
See theorem 1.11, page 244,  in \cite{Els98} for a proof.
Theorem (\ref{conje}) implies
\begin{equation}\label{Acerca3.5}
 \text{Res} \, \big( E_{km}^l (P,s), s=1 \big)    \; = \;
\left\{
	\begin{array}{ll}
		\tfrac{1}{[\Gamma_{\infty} : \Gamma'_{\infty} ]} \cdot \frac{ \abs{\Lambda_D} }{ \text{Vol}(M_D) }   & \mbox{if } l=0 \\
		0            & \mbox{if } l > 0.
	\end{array}
\right.
\end{equation}

From (\ref{Acerca2}), (\ref{Acerca3}) and (\ref{Acerca3.5}) we conclude:
\begin{equation}\label{Acerca4}
 \text{Res} \, \big( H_{km}^l (P,s), s=1 \big)   \; = \;
\left\{
	\begin{array}{ll}
		\tfrac{1}{[\Gamma_{\infty} : \Gamma'_{\infty} ]} \cdot  \frac{ \abs{\Lambda_D} }{ \text{Vol}(M_D) }   & \mbox{if } l=0 \\
		0            & \mbox{if } l > 0.
	\end{array}
\right.
\end{equation}

The formula that relates $\mathfrak{M}(f,s)$ with the series $H_{km}^l (P,s)$ is given by proposition  (\ref{PImp})
\begin{equation}\label{Acerca6}
    \mathfrak{M}(f,s)   = \frac{[\Gamma_{\infty} : \Gamma'_{\infty} ]}{ \abs{\Lambda_D} } \sum_{l=0}^\infty  \sum_{m=-l}^l \sum_{k=-l}^l Y_m^l ( e_1 )
                 \int_{F_D}   \widehat{f}_k^{\, l} ( P) \;  H_{km}^l (P,s)   \, \frac{dx dy d\lambda}{\lambda^3},
                 \; \; \; \Re(s) > 1.
\end{equation}

\bigskip
Since the Mellin transform $\mathfrak{M}(f,s)$ admits a meromorphic continuation to all of $\C$ in the variable $s$,  the formula (\ref{Acerca6}) is valid for all $s \in \C$. Taking residues at $s=1$, we obtain: 
\begin{align}
\text{Res} \, \big( \mathfrak{M}(f,s) , s=1 \big)  & \; = \;  \frac{[\Gamma_{\infty} : \Gamma'_{\infty} ]}{ \abs{\Lambda_D} }
      \sum_{l=0}^\infty  \sum_{m=-l}^l \sum_{k=-l}^l Y_m^l ( e_1 )
                 \int_{F_D}   \widehat{f}_k^{\, l} ( P) \;
          \text{Res} \, \big( H_{km}^l (P,s) , s=1 \big)  \, \frac{dx dy d\lambda}{\lambda^3}      &&  \nonumber \\
             & \; = \;  \frac{[\Gamma_{\infty} : \Gamma'_{\infty} ]}{ \abs{\Lambda_D} } \,  Y_0^0 ( e_1 )
                 \int_{F_D}   \widehat{f}_0^{\, 0} ( P) \; \frac{1}{[\Gamma_{\infty} : \Gamma'_{\infty} ]} \;
          \frac{ \abs{\Lambda_D} }{ \text{Vol}(M_D) }  \, \frac{dx dy d\lambda}{\lambda^3},
          \hspace{1.8cm} \text {by (\ref{Acerca4})}        && \nonumber
\end{align}
but $ Y_0^0  (\vartheta, \varphi)  = \tfrac{1}{2} \, \tfrac{1}{\sqrt{\pi}}$, then
\begin{equation}\label{Acerca7}
  \text{Res} \, \big( \mathfrak{M}(f,s) , s=1 \big) \; = \; \frac{ 1 }{ 2 \sqrt{\pi} \, \text{Vol}(M_D) }
                \int_{F_D}   \widehat{f}_0^{\, 0} ( P) \, \frac{dx dy d\lambda}{\lambda^3}.
\end{equation}

By (\ref{LP8})
\begin{align}
 \widehat{f}_0^{\, 0} (P)  & \; = \;  \frac{1}{ 2 \sqrt{\pi}} \int_{\S^2} f \big( P, \vec{v}( \vartheta, \varphi )_P   \big)
       \,  \sin \vartheta \, d \vartheta  d \varphi.            && \label{Acerca8}
\end{align}

Substituting (\ref{Acerca8}) in (\ref{Acerca7}),
\begin{align}
\text{Res} \, \big( \mathfrak{M}(f,s) , s=1 \big)    & \; = \;  \frac{ 1 }{ 4 \pi \text{Vol}(M_D) }
                \int_{F_D}     \int_{\S^2}
      f \big( P, \vec{v}( \vartheta, \varphi )_P   \big) \;  \frac{\sin \vartheta \,  dx  dy  d\lambda  d\vartheta  d\varphi}{\lambda^3} \; = \;  \nu (f). &&  \nonumber
\end{align}
\end{proof}


\begin{lemm}\label{LemasPrevios8}
If $f \in C_{0}^\infty(T_1 (M_D))$, then:
\[
       \nu (\lambda) (f) \; = \; \frac{1}{2\pi i} \, \int_{2-i \infty}^{2+i \infty} \mathfrak{M} (f,s+1) \, \lambda^{-s} \, ds
          \; = \; \frac{1}{2\pi i} \, \int_{3-i \infty}^{3+i \infty} \mathfrak{M} (f,s) \, \lambda^{1-s} \, ds.
\]
\end{lemm}

\begin{proof}
  Let $\rho : \R^{+} \longrightarrow \C$ given by
\[
   \rho  (\lambda) \; := \; \nu (\lambda) (f), \; \; \; \lambda >0.
\]

\bigskip
Since $f$ has compact support $\rho$ is zero when $\lambda \longrightarrow \infty$.
The Mellin transform $\rho$ is given by
\begin{align}
 \tilde{\rho}  (s)  & \; := \; \int_{0}^{\infty}  \rho  (\lambda) \, \lambda^{s-1} \, d\lambda
        \; = \;   \mathfrak{M}(f,s+1). && \text {by definition (\ref{Mellinf})} && \nonumber 
\end{align}

By the Mellin inversion theorem,
\begin{align}
\nu (\lambda) (f) & \; = \; \frac{1}{2\pi i} \int_{2-i \infty}^{2+i \infty} \tilde{\rho}(s)  \, \lambda^{-s} \, ds
     \; = \;  \frac{1}{2\pi i} \int_{2-i \infty}^{2+i \infty} \mathfrak{M}(f,s+1) \, \lambda^{-s} \, ds.
                             && \nonumber
\end{align}

The vertical line $\Re(s)=2$ can be changed to the vertical line $\Re(s)=3$ making a change of variable  $s+1 =u$, then $ s = 2 \mp i \infty \; \Longleftrightarrow \; u = 3 \mp i \infty$.
\end{proof}

\begin{lemm}\label{LemasPrevios10}
Let $f \in C_{0}^{\infty} \big( T_1 (M_D) \big)$, the following equality hold:
\[
     \int_{-\infty}^{\infty}  \abs {\mathfrak{M} (f,it) }  \, dt \; < \; \infty.
\]
\end{lemm}

\begin{proof}
If $t \neq 0$ lemma (\ref{LemasPrevios6}) implies
\[
       \abs { \mathfrak{M} (f,it) } \; \leq \; \frac{ c_f}{(t^2+1)^2} \;  \norm{ E^T ( \cdot, it ) }^2,
\]
where $c_f \; := \; \tfrac{\norm{ \Delta^2 f }^2}{ \abs{\Lambda_D} }$ and $T=T(f)$ is a  sufficiently large  constant. By formula (4.29), page 290, in  \cite{Els98}, we know that
\begin{equation*}
       \norm{ E^T ( \cdot, it ) }^2 \; = \; \int_{F_D}  \abs{ E^T (z+\lambda j,it) }^2 \, \frac{dx dy d\lambda}{ \lambda^3} \; = \;
       c_\infty \, w(t) \, + \, h(t),
\end{equation*}
where $c_\infty$ is a constant,  $w,h: \R \longrightarrow \R$ are functions  satisfying

\begin{equation}\label{Prev10M3}
       w(t) \; \geq  \; 1 \; \; \; \forall t \in \R.
\end{equation}
\vspace{-0.3cm}
\begin{equation}\label{Prev10M7}
     \displaystyle  \int_{-R}^R  w(u) du  \; = \; \text{O} (R^3) \; \; \text{when } R \rightarrow \infty.
\end{equation}
\begin{equation}\label{Prev10M4}
      h(t) \; = \; \text{O}(1) \; \; \text{when } \abs{t} \rightarrow \infty.
\end{equation}

\bigskip
Therefore,
\begin{equation}\label{Prev10M5}
     \displaystyle \int_{-\infty}^\infty   \abs {\mathfrak{M} (f, it) } \, dt  \; \leq \;
            c_f \,  c_\infty   \int_{-\infty}^\infty  \frac{w(t)}{(t^2+1)^2} \, dt \, + \,
            c_f    \int_{-\infty}^\infty  \frac{h(t)}{(t^2+1)^2} \, dt.
\end{equation}

\bigskip
In the first integral in the right hand side in (\ref{Prev10M5}), we obtain  integrating by parts,
\begin{equation}\label{Prev10M6}
     \displaystyle  \int_{-x}^x  \frac{w(t)}{(t^2+1)^2} \, dt  \; = \; \frac{1}{(x^2+1)^2} \, \int_{-x}^x  w(u)  du
          \, + \, 4 \,  \int_{-x}^x  \Bigg[ \int_0^t  w(u)  du     \Bigg] \, \frac{t}{(t^2+1)^3} \,  dt.
\end{equation}
using the properties (\ref{Prev10M3})  and (\ref{Prev10M7}) of the function $w(t)$ and taking the limit $x \rightarrow \infty$ in (\ref{Prev10M6}) one sees that
\begin{equation}\label{Prev10M13}
     \int_{-\infty}^{\infty}  \frac{w(t)}{(t^2+1)^2} \, dt \;    <  \; \infty.
\end{equation}

From (\ref{Prev10M4}) it follows that
\begin{equation}\label{Prev10M17}
   \int_{-\infty}^\infty  \frac{h(t)}{(t^2+1)^2} \, dt \; < \; \infty.
\end{equation}

\bigskip
Since the integrals in  (\ref{Prev10M13}) and (\ref{Prev10M17}) are bounded by (\ref{Prev10M5}) we  obtain:
\[
     \displaystyle \int_{-\infty}^\infty   \abs {\mathfrak{M} (f, it) } \, dt \,  \; < \; \infty.
\]
\end{proof}


\begin{lemm}\label{LemasPrevios11}
Let $f \in C_{0}^{\infty} \big( T_1 (M_D) \big)$, then:
\[
\frac{1}{2\pi i} \int_{-i \infty}^{+i \infty} \mathfrak{M} (f,s) \, \lambda^{1-s} \, ds  \; = \; o ( \lambda )   \; \; \; \text{when }
       \lambda \rightarrow 0.
\]
\end{lemm}

\begin{proof}
By lemma (\ref{LemasPrevios10}) we can apply Riemann-Lebesgue lemma to the function $\mathfrak{M} (f,it)$, so that
\[
     \displaystyle  \lim_{u \rightarrow \infty} \int_{-\infty}^\infty \mathfrak{M} (f,it) \, e^{iut}  \, dt \; = \; 0.
\]

Changing variable $u=-log \lambda$, it follows that $e^{iut} =  \lambda^{-it}$, in addition $\lambda\longrightarrow 0 \; \Longleftrightarrow  \; u \longrightarrow \infty$.
One has
\begin{equation}\label{Prev11M2}
    \displaystyle  \lim_{\lambda \rightarrow 0} \int_{-\infty}^\infty \mathfrak{M} (f,it) \, \lambda^{-it}  \, dt \; = \; 0.
\end{equation}

\bigskip
If we consider now the integral in the lemma and the change of variable $s=it$, so that $s = \mp i \infty \; \Longleftrightarrow \; t = \mp \infty$, the integral is transformed as follows:
\begin{align}
\frac{1}{2\pi i} \int_{-i \infty}^{+i \infty} \mathfrak{M} (f,s) \, \lambda^{1-s} \, ds  & \; = \;
          \frac{1}{2\pi i}  \int_{ -\infty}^{+ \infty} \mathfrak{M} (f,it) \, \lambda^{1-it} \, i dt   \; = \;
             o (\lambda)   \; \; \text{when} \, \lambda \rightarrow 0.     &&  \text {by (\ref{Prev11M2})} \nonumber
\end{align}
\end{proof}

\section{Final remarks}
We  have shown that if we consider the one-parameter family of normalized Lebesgue measures
$\left\{\nu(\lambda)\right\}_{\lambda>0}$  given by (\ref{nulambda}) where $\nu(\lambda)$ supported on the Euclidean orbifold $\widetilde{T}_\lambda$ then there exists $r>0$ such that for every $\varepsilon>0$,  it holds for each smooth function with compact support $f:T_1(M_D) \longrightarrow \C$:
\[
\nu(\lambda)(f)=\nu(f)+0(\lambda^{r-\varepsilon})\quad\quad \text{as}\,\, \lambda\longrightarrow 0.
\]
A fundamental observation first by Don Zagier and then by Peter Sarnak is: {\bf finding the optimal $r$ is equivalent to the Riemann hypothesis for the Dedekind zeta function $\zeta_D$.}   
\begin{rem} The frame bundle $\mathcal{F}({M}_D)$, to which we referred in the introduction, fibers over the unit tangent bundle $T_1(M)$ with fibre the circle $\bS^1$ i.e., there exists a principal
circle fibration $p:\mathcal{F}({M}_D)\longrightarrow {T_1(M_D)}$. We can lift canonically the stable and unstable
foliations $\cF^{ss}$ and  $\cF^{uu}$ to $\mathcal{F}({M}_D)$ to obtain the complex horocycle foliations i.e.,
the orbits of the stable an unstable horocycle flows.  The orbifolds $\widetilde{T}_\lambda$  lift to the elliptic curves
which are the compact orbits of the complex stable horocycle flow. If $\mathbb E$ is one of the elliptic curves
which is a compact orbit of the complex unstable horocycle flow then the family
$\left\{g_{e^{it}}(\mathbb E)\right\}_{t\in\R}$ projects
to the same orbifold $\widetilde{T}_\lambda$ for some $\lambda$. From this by a simple disintegration of measures, we obtain:
\begin{coro}
If $\mathbb{E}$ is one of the elliptic curves which is a compact orbit of the complex unstable horocycle flow
and $\mu(t)$ is the Lebesgue probability measure supported on the elliptic curve $g_t(\mathbb E)$ ($t\in\R$)
then there exists $r>0$ such that for every $\varepsilon>0$,  it holds for each smooth function with compact support $f:T_1(M_D)\longrightarrow \C$:
\[
\nu(\lambda)(f)=\nu(f)+0(\lambda^{r-\varepsilon})\quad\quad \text{as}\,\, \lambda\longrightarrow 0.
\]
\end{coro}

\end{rem}

A possible continuation of the present work is to study non-arithmetical 3-dimensional hyperbolic manifolds of finite
volume and one cusp, for instance the complement of knots in $\S^3$. We know by Thurston's work \cite{Th} that if $K\in\S^3$ is not a satellite knot or a torus knot then the complement of $K$, $M_K:=\S^3\,\setminus\, K$, has a hyperbolic structure of finite volume
and one cusp. The only arithmetic knot is the figure-eight knot \cite{Reid} which corresponds to a finite covering of the orbifold
${M}_{D}= \H^3 /  \widetilde{\Gamma}_{D}$ with $D=3$, \ie the Bianchi group corresponding to the imaginary quadratic field $\Q(\sqrt{-3})$. Another interesting case is that of non-compact hyperbolic 3-manifolds of finite volume
with more than one cusp, for instance complements of hyperbolic links.

\bigskip
\subsection*{Acknowledgements}This paper contains part of the Ph. D. thesis of the first author. He would like to thank his advisor, the second author Alberto Verjovsky for his suggestion to work on this topic, his constant support, and sharing some of his ideas on this subject. Also for the first author (OR) were very important the seminars and advise of Dr. Santiago L\'opez de Medrano, and the help and patience of Dr. Gregor Weingart. He also thanks Dr. Fuensanta Aroca for a fellowship from her PAPIIT grant  IN108216. The second author (AV) benefited from a PAPIIT (DGAPA, Universidad Nacional Aut\'onoma de M\'exico) grant IN106817.
 Both authors would like to thank Professor Paul Garrett for his comments and references regarding Eisenstein series, and the anonymous referee for his/her suggestions.


\begin{thebibliography}{80}


\bibitem{Anosov} Anosov D. V., {\it Geodesic flows on closed Riemann manifolds with negative curvature,\, } Proceedings of the Steklov Institute of Mathematics, No. 90 (1967). Translated from the Russian by S. Feder American Mathematical Society, Providence, R.I. 1969, p.p. iv+235.


%
\bibitem{Bianchi} Bianchi L., {\it Sui gruppi di sostituzioni lineari con coefficienti appartenenti a corpi quadratici immaginari,\, } Math. Ann. 40, No. 3 (1892), p.p. 332-–412.

%
\bibitem{BM} Bowen, R., Marcus, B., {\it Unique ergodicity for horocycle foliations,\, } Israel J. Math 26, No. 1 (1977), p.p. 43--67.
%

%
\bibitem{BruMoto} Bruggeman R.,  Motohashi Y., {\it Sum formula for Kloosterman sums and fourth moment of the Dedekind zeta-function over the Gaussian number field,\, } Functiones et Approximatio Commentarii Mathematici, 31 (2003), p.p. 23--92.

%


%
\bibitem{DB} Dal'Bo F., {\it Geodesic and Horocyclic Trajectories,\, } Translated from the 2007 French original. Universitext. Springer-Verlag London, Ltd., London; EDP Sciences, Les Ulis (2011), p.p. xii+176.
%

%
\bibitem{Da1} Dani S. G., {\it Dynamical systems on homogeneous spaces,\, } Bull. Amer. Math. Soc. 82 (1976), No. 6, p.p 950--952.

%

%
\bibitem{Da} Dani S. G., {\it Invariant measures and minimal sets of horospherical flows,\, } Invent. Math. 64 (1981), No. 2,  p.p. 357--385.
%


%
\bibitem{DS} Dani S.G., Smillie J., {\it Uniform distribution of horocycle orbits for Fuchsian groups,\, } Duke Math. Journal 51 (1984), p.p.  185--194.
%

%
\bibitem{Els98} Elstroedt J., Grunewald F., Mennicke J., {\it Groups Acting  on Hyperbolic Space: Harmonic Analysis and Number Theory,\, } Springer-Verlag, Berlin (1998).
%
%
\bibitem{Samuel} Estala S., {\it Distribution of cusp sections in the Hilbert modular orbifold,\, } Journal of Number Theory, Vol. 155 (2015), p.p.  202--225.
%

%
\bibitem{FF} Flaminio L., Forni G., {\it Invariant distributions and time averages for horocycle flows,\, } Duke Math. Journal 119, No. 3 (2003), p.p. 465--526.
%


%
\bibitem{Fu} Furstenberg H., {\it The unique ergodicity of the horocycle flow,\, } Recent Advances in Topological Dynamics (Proc. Conf., Yale Univ., New Haven, Conn., 1972), Lecture Notes in Math, Vol. 318, Springer, Berlin (1973), p.p.  95–-115.
%

%
\bibitem{Gh} Ghys \'E., {\it Holomorphic Anosov systems,\, } Invent. Math. 119 (1995), No. 3, p.p. 585--614.
%


%
\bibitem{Garrett} Garrett P., {\it Intertwinings among principal series of $\SL_2(\C)$,\, } http://www-users.math.umn.edu/${\sim}$Garrett/m/v/
%

%
\bibitem{GV} Ghys \'E., Verjovsky A., {\it Locally free holomorphic actions of the complex affine group,\, } Geometric study of foliations (Tokyo, 1993), World Sci. Publ., River Edge, NJ, (1994), p.p. 201--217.
%

%
\bibitem{Lok04} Guleska L., {\it  Thesis:  Sum Formula for SL2 over Imaginary Quadratic Number Fields,\, } Utrecht (2004).
%

%
\bibitem{Langlands} Langlands R., {\it On the functional equations satisfied by Eisenstein series,\, } Lecture Notes in Math 544, Springer-Verlag  Berlin (1976).
%

%
\bibitem{MS} Maucourant F., Schapira B., {\it On topological and measurable dynamics of unipotent frame
flows for hyperbolic manifolds,\, } 2017-19. 2017. $\langle${hal-01455739}$\rangle$.
 %

%
\bibitem{MW} Moeglin C.,  Waldspurger J., {\it Spectral Decomposition and Eisenstein series,\, } Cambridge Univ. Press, Cambridge (1995).
%

 %
\bibitem{Ra} Ratner M., {\it On Raghunathan's measure conjecture,\, } Ann. of Math. 134, No. 3 (1991), p.p. 545--607.
 %
\bibitem{Reid} Reid A. W., {\it Arithmeticity of Knot Complements,\, }
 J. London Math. Soc. (2) 43 (1991), No. 1, p.p. 171--184.

%
\bibitem{Sar80}Sarnak P., {\it Asymptotic Behavior of Periodic Orbits of   the Horocycle Flow and Eisenstein Series,\, } Comm. Pure and App. Math., Vol. 34 (1981), p.p.   719--739.
%

%
\bibitem{Smale} Smale S., {\it Differentiable dynamical systems,\, } Bull. Amer. Math. Soc. 73 (1967), p.p. 747--817.
%

%
\bibitem{Stark1} Stark, H. M., {\it A historical note on complex quadratic fields with class-number one,\, } Proc. Amer. Math. Soc. 21 (1969), p.p. 254--255.
%

%
\bibitem{Stark2} Stark, H. M., {\it A complete determination of the complex quadratic fields of class-number one,\, } Michigan Math. J. 14 (1967), p.p. 1--27.
%

%
\bibitem{Starkov} Starkov, A. N., {\it Dynamical Systems on Homogeneous Spaces,\, }  Translated from the 1999 Russian original by the author. Translations of Mathematical Monographs (Book 190). American Mathematical Society, Providence (2000), RI, p.p xvi+243.
%
\bibitem{Th} Thurston, W. P., {\it The Geometry and Topology of 3-Manifolds,\, } lecture notes, Princeton
University (1978).
%
\bibitem{Wigner} Wigner E., {\it Group Theory and its Application to the Quantum Mechanics of Atomic Spectra,\, } Academic Press (1959).
%


%
\bibitem{Zag79} Zagier D., {\it Eisenstein series and the Riemann zeta   function,\, } Automorphic forms, representation theory and arithmetic, Tata     Institute of Fundamental Research, Bombay 1979, Springer-Verlag (1981), p.p.   275--301.
%

\end{thebibliography}
\end{document}